\documentclass[12pt]{article}

\usepackage{amssymb,amsmath,amsfonts,amssymb}
\usepackage{graphics,graphicx,color,hhline}
\usepackage{bbm}
\usepackage{stmaryrd}
 \usepackage{mathtools}

\usepackage{authblk}
\usepackage{euscript}

\def\@abssec#1{\vspace{.05in}\footnotesize \parindent .2in 
{\bf #1. }\ignorespaces} 

\graphicspath{{/EPSF/}{Figures/}}   

\newtheorem{theorem}{Theorem}[section]
\newtheorem{lemma}[theorem]{Lemma}
\newtheorem{proposition}[theorem]{Proposition}
\newtheorem{corollary}[theorem]{Corollary}

\def \Rm {\mathbb R}
\def \Nm {\mathbb N}

\def \Cm {\mathbb C}
\def \Zm {\mathbb Z}

\def \PP {\mathbb P}

\def\un{{\mathbbmss{1}}} 

\newcommand{\eps}{\varepsilon}
\newcommand{\E}{\mathbb E}

\newcommand{\be}{\begin{equation}}
\newcommand{\ee}{\end{equation}}
\newcommand{\bea}{\begin{eqnarray}}
                    \newcommand{\h}{\textrm{h}}
                     \newcommand{\inc}{\textrm{inc}}
\newcommand{\eea}{\end{eqnarray}}
\newcommand{\bee}{\begin{eqnarray*}}
\newcommand{\eee}{\end{eqnarray*}}

\def\fref#1{{\rm (\ref{#1})}}

\newcommand{\calB}{\mathcal B}
\newcommand{\calC}{\mathcal C}

\newcommand{\calL}{\mathcal L}

\newcommand{\calP}{\mathcal P}

\newcommand{\calG}{\mathcal G}

\newcommand{\calA}{\mathcal A}
\newcommand{\calD}{\mathcal D}
\newcommand{\calT}{\mathcal T}
\newcommand{\frakG}{\mathfrak{G}}

\newcommand{\cout}[1]{}

\begin{document}
\title{Scaling limit of fluctuations for high contrast stochastic homogenization of the Helmholtz equation:\\
  second order moments}
\author[,1]{Olivier Pinaud \footnote{pinaud@math.colostate.edu}}
 \affil[1]{Department of Mathematics, Colorado State University, Fort Collins CO, 80523}
\maketitle

\begin{abstract}
  This work is concerned with the high contrast stochastic homogenization of the Helmholtz equation. Our goal is to characterize the second order moments of the scaling limit of the fluctuations of the wavefield. We show that these moments are those of a random wavefield solution to a homogenized Helmholtz equation with a white noise source term and obtain expressions for its variance. Two factors contribute to the white noise: fluctuations in the inverse permittivity of the high contrast inhomogeneities, and fluctuations in their size. This problem is motivated by wave propagation in sea ice, which is a random compositive of ice and pockets of air and brine. The analysis hinges on three ingredients: a covariance formula due to Chatterjee for functions of independent random variables; small-volume expansions to quantify the fluctuations due to one inclusion; and the standard two-scale expansions for stochastic homogenization.
  \end{abstract}
\section{Introduction}

We consider in this work the following wave propagation problem: an incoming wave, say a plane wave, is impinging on a bounded, simply connected domain $\calB \subset \Rm^3$, the latter consisting of a smooth background and of randomly located inclusions with large constrast. Assuming the wavelength $\lambda $ is both much larger than the typical size $\ell$ of the inclusions and much smaller than the diameter $d$ of $\calB$, the wave is, in first approximation, reflected and transmitted according to Snell-Descartes laws at the boundary of $\calB$. This is a double consequence of geometrical optics and homogenization: since $\lambda \gg \ell$, and provided the random medium satisfies appropriate conditions, the (stochastic) homogenization regime holds and the random medium is seen by the wave as an effective, constant and deterministic medium; the high frequency condition $d \gg \lambda$ then allows for the use of Snell-Descartes laws.

The situation described here is motivated by the sea ice problem, where a plane flying over sea ice and equipped with a radar system emits pulses probing the sea ice, in particular its depth. Sea ice is a composite of an ice background and random pockets of air and brine, the latter having a large permittivity compared to the background. See \cite{JP-DCDS} for experimental values. Since the pulses impinge on the sea ice with a non-normal incidence, the leading contribution to the wavefield is scattered away by the sea ice surface, so that the backscattering measured at the plane is very small and consists of corrections to the homogenization and high frequency asymptotic regimes. The objective of this work is to characterize the corrections to the high contrast stochastic homogenization approximation.

In order to focus on the corrector problem, we model for simplicity wave propagation by the Helmholtz equation rather than with Maxwell's equations. We borrow the high contrast homogenization setting of \cite{bouchitte2015resonant}: with $\eta=\ell/\lambda \ll 1$, we denote by $a_\eta$ the inverse relative permittivity in the entire space $\Rm^3$ and by $u_\eta$ the magnetic field. The rescaled Helmholtz equation, equipped with Sommerfeld radiation condition at infinity, then reads
 \begin{equation} \label{eqhelm}
\left\{
  \begin{array}{l}
    \nabla \cdot (a_\eta \nabla u_\eta) + k_0^2 u_\eta=f, \qquad \textrm{in  } \Rm^3\\
    (\partial_r-ik_0)(u_\eta-u_{\rm{inc}})=O(\frac{1}{r^{2}}), \qquad \textrm{as} \quad r=|x| \to \infty
  \end{array}
  \right.
  \end{equation}
  where $u_{\rm{inc}}$ is the incoming wave and is such that $\Delta u_{\rm{inc}}+k_0^2u_{\rm{inc}}=f$ over $\Rm^3$, with $f$ a source term supported in $\calB^c=\Rm^3\backslash \overline{\calB}$ ($\overline{\calB}$ denotes the closure of the open domain $\calB$). The number $k_0>0$ is the rescaled nondimensional wavenumber and is of order one. As is customary in probability theory, we will use the variable $\omega$ to denote a realization of the random medium, so that in this paper $\omega$ is not the angular frequency of the wave. We denote by $\calD_\eta(\omega)$ the random domain occupied by the inclusions and by $\eta^2 a_{\eta,\inc}(x,\omega)$ their inverse relative permittivity (we will give a precise definition of $a_{\eta,\inc}$ in Section \ref{setting}). The factor $\eta^2$ is reminiscent of the large-contrast property and is such that the so-called optical diameter of the inclusions is independent of $\eta$, see \cite{bouchitte2015resonant}. We then have
  $$
  a_\eta(x,\omega)=a(x)\un_{\Rm^3 \backslash \overline{\calD_\eta}(\omega)}(x)+\eta^2 a_{\eta,\inc}(x,\omega)\un_{\calD_\eta(\omega)}(x),
  $$
  where $\un_A$ denotes the characteristic function of a set $A$. Above, $a(x)$ is the background inverse relative permittivity, equal to one for simplicity outside of $\calB$ and to $a_b$ in $\calB$, that is $a(x)=1$ for $x\in \Rm^3\backslash \overline{\calB}$ and $a(x)=a_b$ for $x \in \calB$. A depiction of the random domain is given in Figure \ref{fig:domain} further.

  Standard homogenization theory yields that $u_\eta$ converges in appropriate sense as $\eta \to 0$ to the homogenized field $u_{\rm{h}}$ solution to
   \begin{equation} \label{eqhomo}
\left\{
  \begin{array}{l}
    \nabla \cdot (A_{\rm{eff}}(x) \nabla u_{\textrm{h}}) + k_0^2 \mu_{\rm{eff}}(x)u_{\rm{h}}=f, \qquad \textrm{in  } \Rm^3\\
    (\partial_r-ik_0)(u_{\rm{h}}-u_{\rm{inc}})=O(\frac{1}{r^{2}}), \qquad \textrm{as} \quad r=|x| \to \infty.
  \end{array}
  \right.
\end{equation}
Above, $A_{\rm{eff}}(x)=1$, $\mu_{\rm{eff}}(x)=1$ for $x\in \Rm^3\backslash \overline{\calB}$, and $A_{\rm{eff}}(x)=A_{\rm{eff}} \in \Rm^{3 \times 3}$, $\mu_{\rm{eff}}(x)=\mu_{\rm{eff}} \in \Cm$ for $x \in \calB$. The constant coefficients $A_{\rm{eff}}$ and $\mu_{\rm{eff}}$ are the inverse relative permittivity and permeability of the effective medium of propagation. Their expressions are given in Section \ref{setting}. Note the presence of the artificial magnetic permeability $\mu_{\textrm{eff}}$ which has interesting practical applications \cite{felbacq2005negative}. The homogenization limit is established rigorously in \cite{bouchitte2015resonant} in the two dimensional case using stochastic two-scale convergence. See as well \cite{cherdantsev2019stochastic} for the stochastic homogenization of elliptic high contrast operators. 

We are interested in this work in the rescaled random fluctuation $v_\eta=(u_\eta-\E\{u_\eta\})/\eta^{3/2}$ as $\eta \to 0$, which is the quantity measured at the plane since the backscattered part in $\E\{u_\eta\}$ is negligible as claimed earlier. While the stochastic homogenization theory is now well-established, a theory of the fluctuations remained an outstanding problem for years and was only successfully developped by various authors in the last decade, see \cite{armstrong2016quantitative,armstrong2017additive,armstrong2019quantitative,gu2016scaling,duerinckx2020structure,duerinckx2020robustness,gloria2015corrector}. We focus in this article on the second order moments of $v_\eta$, which is already a quite difficult problem, even informally, while the full statistical characterization of $v_\eta$ will be addressed in a separate work. We remain at an informal level, a rigorous mathematical theory is extremely involved and beyond the scope of this work where our goal is to identify the limit of $v_\eta$ and the associated stochastic PDE that can be used in turn for the resolution of inverse problems.

It is expected that $v_\eta$ is asymptotically characterized by a Gaussian field with a covariance function that we explicitly derive in this work (the limit is related to variants of Gaussian free fields, see e.g. \cite{armstrong2019quantitative}). We obtain that the limit of the variance of $v_\eta$ as $\eta \to 0$ is equal to the variance of a random field $v$ which verifies the following SPDE, equipped with homogeneous Sommerfeld radiation conditions,
\be  \label{eqv}
                 \nabla_x \cdot \big(A_{\rm{eff}}(x)\nabla_x v\big) + k_0^2 \mu_{\rm{eff}}(x)v=-a_b \nabla_x \cdot \big(\un_{\calB}W^\theta\nabla_x u_{\textrm{h}}\big)-k_0^2 \un_{\calB}(N^\theta+N^a)u_{\textrm{h}}, \quad \textrm{in  } \Rm^3, 
                 \ee
                 where $W^\theta$ is a matrix-valued (Gaussian) white noise, and $N^\theta$, $N^a$ are scalar (Gaussian) white noise.
                 
                 There are different approaches to study the fluctuations $v_\eta$, and we will follow that of \cite{gu2016scaling}. The method is based on primarily two ingredients: (i) a formula for the covariance of functions of random variables, and (ii) two-scale expansions of $u_\eta$. In the context of Gaussian random fields as in \cite{gu2016scaling}, the covariance formula is given by the Helffer-Sj\"ostrand formula. It does not apply here in our setting of high contrast homogenization. We will then instead use a formula introduced by Chatterjee in \cite{chatterjee2008new} that applies to functions of independent random variables (the inclusions are indeed independent). In addition to these two ingredients, we will need small volume expansions in the spirit of \cite{ammari2007polarization,AmMoskVog02,CMV-IP,BP-AA} to characterize the effect of one inclusion on the  wavefield. Establishing the Gaussian property of the limiting field $v$ will be treated elsewhere and requires additional probabilistic tools such as the so-called second order Poincar\'e inequality \cite{chatterjee2008new}.

                 A similar question as the one treated in this paper is addressed in \cite{JP-DCDS} in a much simpler setting where the randomness is in the refractive index and therefore in the lower order term in the Helmholtz equation. A fully rigorous characterization of the scaling limit is possible in this case and is based on iterated integrals.


                 The paper is structured as follows: we introduce the homogenization setting, two-scale expansions and correctors in Section \ref{setting}; Section \ref{mainresult} is dedicated to our main result and Section \ref{seccov} to the covariance formula. Small-volume expansions are derived in Section \ref{secSV}, and the asymptotic expression of the variance is obtained in Section \ref{secvar}. A formal derivation of the two-scale expansion is proposed in an Appendix, as well as results of Nash-Aronson type for Green's functions in perforated domains used in the proof.
                 \paragraph{Acknowledgment.} The author acknowledges support from NSF grant DMS-2006416 and would like to thank Pr. Margaret Cheney for introducing him to the sea ice problem. 

\section{Homogenization setting} \label{setting}

\subsection{Probabilistic framework}
We follow almost verbatim the setting of \cite{bouchitte2015resonant}. The homogenization result holds when the inclusions do not overlap, which is achieved by partitioning $\Rm^3$ into identical cubes and by placing randomly in each cube an inclusion at a non-zero distance to the boundary. Each inclusion is associated with an index $j \in \Zm^3$ and is characterized by its position $\theta_j$ in the unit cube $Z=(0,1)^3$, by its radius $\rho_j \in [0,1/2]$, and by its inverse permittivity $a_j \in \Cm^+$, for $\Cm^+$ the set of complex numbers with a strictly positive imaginary part (the permittivity of brine is indeed in $\Cm^+$).

The probabilistic framework is the following. For $\xi$ a fixed parameter such that $0<\xi <1/2$, let
$$
S_0=\left\{ (\theta,\rho) \in Z \times [0,1/2]: \textrm{dist}(\theta,\partial Z)>\rho+\xi \right\}.
$$
Let in addition $S=S_0 \times \Cm^+$, $\Omega=S^{\Zm^3}$, and the sample space $\widetilde \Omega$ is defined by $\widetilde{\Omega}= \Omega \times Z$. We denote by $\nu_0$ and $\nu_1$ Borel probability measures on $S_0$ and $\Cm^+$ such that $\nu_1$ has a compact support. We set $\nu=\nu_0 \otimes \nu_1$. An element $\omega \in \widetilde{\Omega}$ then reads $\omega=\big((\theta_j,\rho_j,a_j)_{j \in \Zm^3},z\big)$. We will use the shorthand $m=(m_j)_{j \in \Zm^3}$, with $m_j=(\theta_j,\rho_j,a_j)$, and will denote by $m_j(\omega)$ and $z(\omega)$ the components associated with an event $\omega=(m,z)$. Denoting by $\calL$ the Lebesgue measure on $\Rm^3$, we introduce the probability measures $\PP=\otimes_{\Zm^3} \nu$ on $\Omega$ and $\calP= \PP \otimes \calL$ on $\widetilde \Omega$, with event space $\calA=\otimes_{\Zm^3} \calT(S)$ for $\PP$ and $ \calA \otimes \calT(Z)$ for $\calP$. In the previous statements, $\calT(A)$ denotes the Borel sigma-algebra generated by the set $A$.  With these definitions, the triples $(\theta_j(\omega),\rho_j(\omega),a_j(\omega))$ are independent and identically distributed. For the corrector theory, we will actually mostly work with the probability space $(\Omega,\calA,\PP)$ and not with $(\widetilde{\Omega},\calA \otimes Z,\calP)$ since the variable $z$ will be fixed. The symbol $\E$ will denote both expectation w.r.t. $\PP$ or w.r.t. $\PP \otimes \PP$ when doubling the variables for the covariance formula further. Expectation w.r.t. $\calP$ is denoted $\E_\calP$.

Homogenization occurs under the conditions that the random field is stationary and ergodic. To this end, let, for $x \in \Rm^3$, $T_x : \widetilde{\Omega} \to \widetilde{\Omega}$ be the transformation defined by
$$
T_x \omega=T_x \big((m_j)_{j \in \Zm^3},z\big)=\big((m_{j+[z+x]})_{j \in \Zm^3},z+x-[z+x]\big),
$$
where, with an abuse of notation, $[x]$ is a vector containing the greatest integer parts (i.e. the floor function) of the coordinates of $x$. The map $(x,\omega)\to T_x(\omega)$ preserves the measure $\calP$ and is ergodic \cite{bouchitte2015resonant}.

We now introduce the random domain. We denote by $B(\theta,\rho)$ the open ball centered at $\theta$ of radius $\rho$. For $j \in \Zm^3$, let $D^j(\omega)=B(j-z(\omega)+\theta_j(\omega),\rho_j)$, which can be recast as
$$D^j(\omega)=\left\{x\in \Rm^3: [x+z(\omega)]=j\quad \textrm{and} \quad  x-j+z(\omega) \in B(\theta_0(T_x\omega),\rho_0(T_x \omega))\right\}.$$
We denote by $D_\eta^j(\omega)=\eta D^j(\omega)$ the rescaled version of $D^j(\omega)$. 
The random domain occupied by the inclusions is then
  $$
  D_\eta(\omega)=\bigcup_{j \in J_\eta(\omega)}D_\eta^j(\omega)
  $$
  where
  $$J_\eta(\omega)=\left\{ j \in \Zm^3:\; \eta(j-z(\omega)+Z) \subset \calB\right\}.
  $$
  Note that we considered spherical inclusions for simplicity, but the theory can be extended almost directly to other shapes, see e.g. \cite{cherdantsev2019stochastic}. With such definitions, the inverse relative permittivity of the $j-$th inclusion is then $a_{\eta,\inc}(x,\omega)=a_j(\omega)=a_0(T_x \omega)$ when $x \in D_\eta^j(\omega)$, and the overall random field reads
 $$
  a_\eta(x,\omega)=a(x)\un_{D_\eta^c(\omega)}(x)+\sum_{j \in J_\eta(\omega)} \eta^2 a_j(\omega)\un_{D_\eta^j(\omega)}(x)
  $$
  where $D_\eta^c(\omega)=\Rm^3 \backslash \overline{D_\eta}(\omega)$. A 2D version of the random domain is depicted in Figure \ref{fig:domain}.
  
  \begin{figure}[h!]
\begin{center}
\includegraphics[scale=0.25]{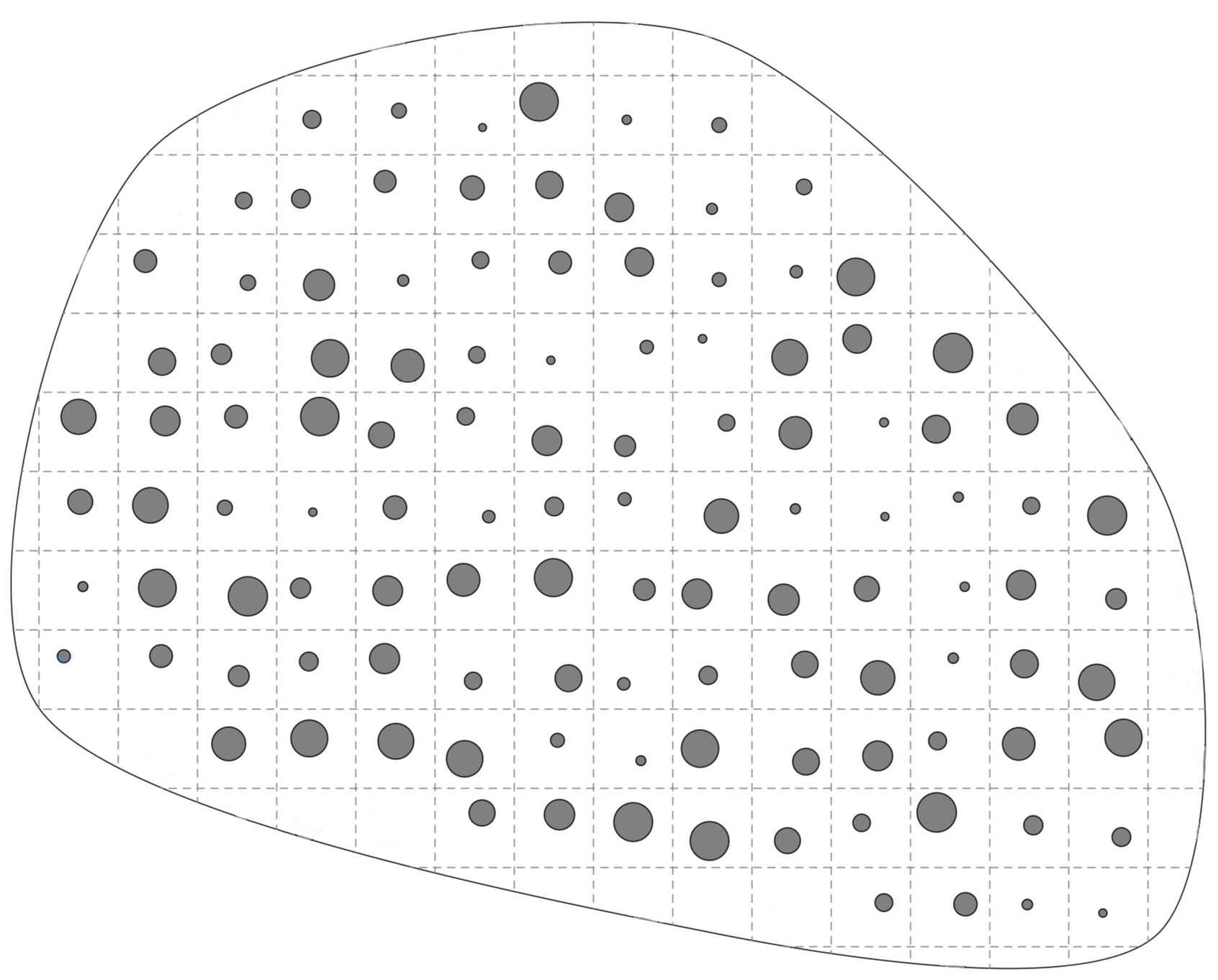} 
\end{center}
\caption{Random domain in 2D. The side of the squares is $\eta$, and only cells fully enclosed in $\calB$ contain inclusions. The background has inverse permittivity of order $O(1)$, while that of the inclusions is of order $O(\eta^2)$. }
\label{fig:domain}
\end{figure}
\subsection{Two-scale expansion and correctors} An important part of the analysis relies on the two-scale expansion of the solution $u_\eta$ to \fref{eqhelm}. With $D(\omega)=\bigcup_{j \in \Zm^3} D^j(\omega)$, the  two-scale expansion for $u_\eta$ is
    \be \label{2sca}
    u_\eta(x,\omega)=u_0(x,x/\eta,\omega)+\eta u_1(x,x/\eta,\omega)+O(\eta^2).
    \ee
    The informal derivation of \fref{2sca} is classical and is detailed in Appendix \ref{2scale}. With $y=x/\eta$, the leading term $u_0$ reads $u_0(x,y)=u_{\rm{h}}(x) \Lambda(y,\omega)$, where $u_{\rm{h}}$ is the homogenized solution satisfying \fref{eqhomo} and $\Lambda(y,\omega)$ is a random function solution to
    \be \label{Lamb}
    a_j(\omega) \Delta_y \Lambda(y,\omega)+k_0^2 \Lambda(y,\omega)=0 \qquad \textrm{on} \quad D^j(\omega),\qquad j \in \Zm^3,
    \ee
    with the boundary condition $\Lambda =1$ on $\partial D^j(\omega)$, and $\Lambda$ extended to 1 on $D(\omega)^c$. The function $\Lambda$ is stationary in the sense that $\Lambda(y+z,\omega)=\Lambda(y,T_z \omega)$, $\forall z \in \Rm^3$. 

    The first-order term $u_1(x,y,\omega)$ reads $u_1(x,y,\omega)= \phi(y,\omega) \cdot \nabla u_{\rm{h}}(x)$, where the corrector $\phi$ is given by $\phi=(\phi_1,\cdots,\phi_3)$ with
$$
      \Delta_y \phi_i=0 \qquad \textrm{on} \quad D(\omega)^c,
$$
      equipped with the boundary conditions
      $$
      n^j \cdot \nabla_y \phi_i=- e_i \cdot n^j \qquad \textrm{on} \qquad  \partial D^j(\omega).
      $$
      Above, $n^j$ is the outward normal to $\partial D^j(\omega)$ and $e_i$ is a vector of the canonical basis of $\Rm^3$. 
     The corrector $\phi$ is defined rigorously following the theory of homogenization in perforated domains, see e.g. \cite{jikov2012homogenization}, Chapter 8. Consider first the set
      $$
      \Sigma=\big\{\omega \in \widetilde{\Omega}: |\theta_0(\omega)-z(\omega)|<\rho_0(\omega)\big\}.
      $$
It can be checked that the domain $D(\omega)$ defined earlier can be expressed as $D(\omega)=\{x \in \Rm^3: T_x \omega \in \Sigma\}$. For $i=1,\cdots,3$ and $y \in D(\omega)$, the gradient $\nabla \phi_i (y,\omega)$ has to be replaced by the stationary vector field $v_i(T_y\omega)$, that satisfies
      $$
      \nabla_y \cdot \left(\un_{D(\omega)^c} (e_i+v_i(T_y \omega))\right)=0.
      $$
      The vector field $v_i$ is such that there exists $v_i^\delta(\omega)$ with $\E_\calP\{|v_i^\delta|^2\}<\infty$, $\E_\calP\{v_i^\delta\}=0$, $\nabla_y \times v^\delta (T_y \omega)=0$ for all $y \in \Rm^d$, and such that
      \be \label{limv}
      \lim_{\delta \to 0} \int_{\Sigma^c} |v_i(\omega)-v_i^\delta(\omega)|^2 d\calP(\omega)=0.
      \ee
      Note that $v_i$ is only defined on $\Sigma^c$ and that, according to Birkhoff ergodic theorem, see \cite{jikov2012homogenization}, Theorem 7.2, \fref{limv} and $\E_\calP\{|v_i^\delta|^2\}<\infty$, we have the growth condition
  $$    \lim_{R \to \infty} \frac{\int_{B_R } \un_{D(\omega)^c}(y)|v_i (T_y \omega)|^2 dy}{|B_R|}< + \infty,
$$
      where $B_R$ is the ball of radius $R$ centered at $0$ and $|B_R|$ its volume. Denote by $\phi^\delta_i(y,\omega)$ a potential function for $v_i^\delta(T_y \omega)$ (which exists since $\nabla_y \times v^\delta (T_y \omega)=0$ for all $y \in \Rm^3$). The condition \fref{limv} shows that, for all $y \in D^c(\omega)$,
      \be \label{lim2E} \lim_{\delta \to 0} \E_\calP |v_i(T_y \omega)-\nabla \phi_i^\delta(y,\omega)|^2=0. \ee
      Since all calculations in this work are formal, we will keep using for simplicity $\nabla \phi_i(y,\omega)$ for $v_i(T_y \omega)$, instead of working with $\nabla \phi_i^\delta$ and repeatedly invoking \fref{lim2E} to pass to the limit.
      %
      

With $\Lambda$ and $\nabla \phi$ at hand, the effective coefficients have the following expressions. We have first
         $$
      A_{\rm{eff}}=a_b \int_{\Sigma^c}(I+(v_{1}(\omega),\cdots,v_{3}(\omega))d\calP(\omega),
      $$
where $I$ is the $3 \times 3$ identity matrix. The matrix $A_{\rm{eff}}$ is symmetric positive definite but in general not proportional to the identity \cite{bouchitte2015resonant}. Besides, equation \fref{Lamb} can be solved exactly by projection on the eigenfunctions $(u_n)_{n \in \Nm}$  of the Dirichlet Laplacian on the unit sphere $S_2$. Denoting by $(\lambda_n)_{n \in \Nm}$ the corresponding (nonpositive) eigenvalues, we find
      \bee
      \Lambda(y,\omega)&=&1-\sum_{j \in \Zm^3} \un_{D^j(\omega)} (y)\sum_{n \in \Nm}  \frac{c_n k_0^2 \rho_j^2(\omega)}{k_0^2\rho_j^2(\omega)+a_j(\omega) \lambda_n} u_n \left(\frac{y+z(\omega)-\theta_j(\omega)}{\rho_j(\omega)}\right),
      \eee
      with $c_n=\int_{S_2} u_n dy$. The function $\Lambda$ is stationary, and as a consequence its expectation does not depend on $y$. A direct calculation finally shows that 
      $$
      \mu_{\rm{eff}}=\E_\calP\{\Lambda(y,\cdot)\}=1-\sum_{n \in \Nm} c_n^2 k_0^2 \E_\calP\left\{\frac{ \rho^4(\omega)}{k_0^2\rho^2(\omega)+a(\omega) \lambda_n}\right\}.
      $$

The functions $\Lambda$ and $\phi$ are sufficient for the homogenization theory. For the scaling limit, we  need in addition the following modified corrector $\phi^{j,-}=(\phi_1^{j,-},\cdots,\phi_3^{j,-})$, obtained by removing an inclusion in the cell $j$. For $j \in \Zm^3$, let first
  $$  \calD_j(\omega)=\bigcup_{k \in \Zm^3, k \neq j}D^k(\omega),
$$
    and for $i=1,\cdots,3$, let $\phi^{j,-}_i$ such that 
     \be \label{corr1m}
      \Delta_y \phi_i^{j,-}=0 \qquad \textrm{on} \quad \calD_j(\omega)^c,
      \ee
      with the boundary condition, for $\ell \neq j$,
      \be \label{corr2m}
      n^\ell \cdot \nabla_y \phi_i^{j,-}=- e_i \cdot n^\ell \qquad \textrm{on} \qquad  \partial D^\ell(\omega).
      \ee
     
      The construction of the corrector $\phi$ hinges on the probabilistic framework and the stationarity of its gradient. The gradient of the modified corrector $\phi^{j,-}$ cannot be stationary since there is no inclusion in the cell $j$ and the probabilistic approach cannot be used directly. A simple way to obtain a solution is by perturbation of $\nabla \phi$ and by imposing the condition at the infinity
      \be \label{corr3m}
            \lim_{|y| \to \infty} \un_{\calD_j(\omega)^c}(y)|\nabla \phi_i^{j,-}(y,\omega)- \nabla \phi(y,\omega)|=0.
      \ee

    We show informally in Appendix \ref{phijm} that there is a unique $\phi_i^{j,-}$, up to an additive constant, satisfying \fref{corr1m}-\fref{corr2m}-\fref{corr3m}. Note, as expected, and even though it is not apparent from \fref{corr3m}, that $\phi_i^{j,-}$ does not depend on the characteristics of the inclusion in cell $j$. We have indeed, see further \fref{Dphi}-\fref{DE}, that
      $$
       \lim_{|y| \to \infty} \un_{\calD_j(\omega)^c}(y)|\nabla \phi_i(y,\omega)- \nabla \phi(y,\omega^j)|=0,
       $$
       where the component $j$ in $\omega^j$ is replaced by an independent copy and is therefore independent of $m_j$.
      %
Note finally that the functions $\nabla \phi$ and $\nabla \phi^-$ are real-valued while $\Lambda$ is complex-valued. 

  \section{Main result} \label{mainresult}
  Let $g$ be a real-valued smooth function with compact support and supported in $\calB^c$, and let $(f,g)=\int_{\Rm^3} f^*(x) g(x) dx$ where $f^*$ is the complex conjugate of $f$ (we will also occasionally denote the complex conjugate by $\overline{f}$ when the notation is more appropriate). We are interested in the fluctuation
  $$U_\eta(g)=(g,u_\eta-\E\{u_\eta\})/\eta^{3/2}.
  $$
  Note that expectation is taken w.r.t. to $\PP$, so that $z$ is fixed here. This is motivated by the sea ice problem: it is not possible to estimate $\E_\calP\{u_\eta \}$ by empirical averages since measurements are performed for only one realization of the random variable $z$, while it is possible to estimate $\E\{u_\eta \}$ for one realization of the random field by slicing the random domain into several slabs and invoking the independence of the inclusions. We will see that while $U_\eta(g)$ depends on $z$, the limit of its variance does not which justifies our approach.

  We need to introduce a few more notations before stating our result. For $\omega=(m,0)$ with $m=(m_j)_{j \in \Zm^3}=(\theta_j(\omega),\rho_j(\omega),a_j(\omega))_{j \in \Zm^3}$, we denote for simplicity the corrector $\phi(y,\omega)$ by $\phi(y,m)$. In the same way, the modified corrector with no inclusion in the cell labeled zero, $\phi^{0,-}(y,\omega)$, is denoted by $\phi^-(y,m)$. Note that $\phi^-$ does not depend on $m_0$. For $\ell, k=1,2,3$, let now
 \bea \label{defMF}
 M_{\ell,k}(m)&=&\int_{\partial B(\theta_0(m),\rho_0(m))} n \cdot \big[\nabla y_\ell + \nabla \phi^{-}_\ell(y,m)  \big] \big[y_k+\phi_k (y,m)) \big] dS(y)\\
 N(m)&=&\int_{B(\theta_0(m),\rho_0(m))} \left[1+\Lambda(y,m)\right] dy. \label{defNF}
        \eea
        Above, $dS$ is the surface measure on $\partial B(\theta_0(m),\rho_0(m))$, and $n$ denotes the outward normal to $\partial B(\theta_0(m),\rho_0(m))$. We will see in Section \ref{secvar} that the matrix $M$ is uniquely defined even though $\phi$ is defined up to an additive constant. Note that $M$ is real-valued while $N$ is complex-valued. With $\Zm^3_\star=\Zm^3 \backslash \{0\}$, consider now
        \bee\langle M \rangle_{(\theta_0,\rho_0)}  &=&\int_{S_0^{\Zm^3_\star}} M(m) \otimes_{j \in \Zm^3_\star} d \nu_0(\theta_j,\rho_j) \\
 \langle M \rangle &=&\int_{S_0^{\Zm^3}} M(m) \otimes_{j \in \Zm^3} d \nu_0(\theta_j,\rho_j),
 \eee
 namely $\langle M \rangle$ is the average of $M$ over all $(\theta_j(\omega),\rho_j(\omega))_{j \in \Zm^3}$ keeping $(a_j)_{j \in \Zm^3}$ fixed, while $\langle M \rangle_{(\theta_0,\rho_0)}$ is the average over $(\theta_j(\omega),\rho_j(\omega))_{j \in \Zm^3_\star}$ keeping $(a_j)_{j \in \Zm^3}$ and $(\theta_0, \rho_0)$ fixed. Let furthermore
  \bee 
 \calC^W &=& \E \big \{ (\langle M \rangle_{(\theta_0,\rho_0)} -\langle M \rangle)\otimes (\langle M \rangle_{(\theta_0,\rho_0)} -\langle M \rangle)\big\}\\
 \calC^{W,N} &=& \E \big \{ ( \langle M\rangle -\langle M \rangle_{(\theta_0,\rho_0)})(\langle N \rangle_{(\theta_0,\rho_0)}-\langle N \rangle ) \big\}\\
 \calC^{\theta} &=&  \E \big \{ (\langle N \rangle_{(\theta_0,\rho_0)} -\langle N \rangle)(\langle N \rangle_{(\theta_0,\rho_0)} -\langle N \rangle) \big\}\\
 \calC^{\theta,*} &=& \E \big \{( \langle N \rangle_{(\theta_0,\rho_0)} -\langle N \rangle)(\langle \overline{N} \rangle_{(\theta_0,\rho_0)}-\langle \overline{N} \rangle)\big\}\\
 \calC^{a} &=& \E \big \{ (\E \{ N | a_0 \}-\E\{N\})(\E \{N | a_0 \} -\E\{N\})\big\}\\
  \calC^{a,*} &=&  \E  \big \{( \E \{ N | a_0 \}-\E\{N\})(\E \{\overline{N} | a_0 \}-\E\{\overline{N}\}) \big\}.
 \eee
 Above, the notation $C=A\otimes B$ means $C_{ijk\ell}=A_{ij}B_{k \ell}$ for $i,j,k,\ell=1,2,3$, $\overline{N}$ is the complex conjugate of $N$, and $\E \{ N | a_0 \}$ is the expectation of $N$ conditioned on $a_0$.
Consider now $v$ the solution to \fref{eqv} and let $U(g)=(g,v)$. In \fref{eqv}, the random field $N^a$ is independent of $(W,N^\theta)$, while $W$ and $N^\theta$ are correlated. For the correlation functions defined above, these random fields verify: 
                 \bee
                 \E\{W(x) \otimes W (x')\}&=&\calC^W \delta(x-x')\\
                  \E\{W(x)N^\theta(x')\}&=&\calC^{W,N} \delta(x-x')\\ 
                  \E\{N^\theta (x) N^\theta(x')\}&=&\calC^{\theta} \delta(x-x')\\
                                   \E\{N^\theta (x) \overline{N}^\theta(x')\}&=&\calC^{\theta,*} \delta(x-x')\\
                                   \E\{N^a (x) N^a(x')\}&=&\calC^a \delta(x-x')\\
                                     \E\{N^a (x) \overline{N}^a(x')\}&=&\calC^{a,*} \delta(x-x').
                 \eee
                 
                 We are now in position to state our main result.

                 \begin{theorem} \label{th} We have
                   $$
                   \lim_{\eta \to 0} \E\{|U_\eta(g)|^2\}=\E\{|U(g)|^2\}, \qquad \lim_{\eta \to 0} \E\{(U_\eta(g))^2\}= \E\{(U(g))^2\}.
                   $$
                 \end{theorem}
                 Note that since the limit of $U_\eta(g)$ is expected to be (complex) Gaussian, it is enough to characterize the limits of $\E\{|U_\eta(g)|^2\}$ and $\E\{(U_\eta(g))^2\}$ to fully characterize the limit of $U_\eta(g)$. The limits above can be explicited as follows. Let $G_\h$ be the homogenized Green's function satisfying
       \begin{equation} \label{homoGF}
\left\{
  \begin{array}{l}
    \nabla_x \cdot (A_{\rm{eff}}(x) \nabla_x G_\h(x,x')) + k_0^2 \mu_{\rm{eff}}(x)G_\h(x,x')=\delta(x-x'), \qquad \textrm{in  } \Rm^3\\
    (\partial_r-ik_0)G_\h=O(\frac{1}{r^{2}}), \qquad \textrm{as} \quad r \to \infty.
  \end{array}
  \right.
\end{equation}
Then, $U(g)$ can be expressed as
              \bee
                 U(g)&=&\int_{\Rm^3} v(x) g(x) dx\\
                 &=&a_b\int_{\Rm^3}\int_{\calB} \nabla_{x'} G_{\rm{h}}(x,x') \cdot W(x') \nabla_{x'}u_{\rm{h}}(x') g(x) dx\\
                 &&-k_0^2 \int_{\Rm^3}\int_{\calB} G_{\rm{h}}(x,x') \cdot [N^\theta+N^a](x') u_{\rm{h}}(x') g(x) dx,
                 \eee   
     and

             \bee
             \lim_{\eta\to 0} \E \{|U_\eta(g)|^2\}&=&a_b^2 \sum_{i,j,k,l=1}^3 \calC^W_{ijkl}\int_\calB \partial_{x_i} G_{\textrm{h},g}\partial_{x_j} u_{\textrm{h}} \partial_{x_k} \overline{G}_{\textrm{h},g} \partial_{x_l} \overline{u}_{\textrm{h}} dx\\
             &&+k_0^4\left( \calC^{\theta,*}+ \calC^{a,*}\right)\int_\calB |G_{\textrm{h},g} u_{\textrm{h}}|^2  dx\\
             &&+  a_b k_0^2\sum_{i,j=1}^3  \overline{\calC}^{W,N}_{ij}\int_\calB \partial_{x_i} G_{\textrm{h},g}\partial_{x_j} u_{\textrm{h}} \overline{G}_{\textrm{h},g} \overline{u}_{\textrm{h}}dx\\
             &&+a_b k_0^2\sum_{i,j=1}^3  \calC^{W,N}_{ij}\int_\calB \partial_{x_i} \overline{G}_{\textrm{h},g}\partial_{x_j} \overline{u}_{\textrm{h}} {G}_{\textrm{h},g} {u}_{\textrm{h}}dx.
             \eee
             A similar relation holds for the limit of $\E \{(U_\eta(g))^2\}$.

             The rest of the paper is dedicated to the proof of Theorem \ref{th}. The starting point is an adaptation of a covariance formula due to Chatterjee: the function $u_\eta$ depends on the collection of random variables $(\theta_j,\rho_j,a_j)_{j \in \Zm^3}$, and actually on only a finite number of them associated with inclusions in $\calB$. Since these random variables are independent for different $j$, it is possible to use the results of \cite{chatterjee2008new} after small modifications. This is done in the next section. 
  
\section{Covariance formula} \label{seccov}

Let $n \geq 1$ and $X=(X^\theta,X^a)$ be a $n \times 3$ random matrix, where $X^\theta=(\theta,\rho)$ ($X^\theta$ should actually be denoted $X^{\theta,\rho}$ which we avoided for simplicity), and $X^\theta=(X^\theta_1, \cdots,X^\theta_n )^T \in S_0^n$ ($^T$ denotes transposition) contains $n$ independent realisations of $(\theta,\rho)$ identically distributed according to $\nu_0$. The random vector $X^a=(X^a_1, \cdots,X^a_n )^T \in \Cm_+^n$ contains $n$ independent realisations of $a$ identically distributed according to $\nu_1$, and $X^\theta$ is independent of $X^a$. Note that the variables $\theta$ and $\rho$ are dependent since $(\theta,\rho) \in S_0$, which is why we treat them together in $X^\theta$. 
We denote by $X'$ an independent copy of $X$. Let $[n]=\{1,\cdots,n\}$ and for a subset $A$ of $[n]$, let, for $\alpha=\theta,a$,
  $$
 X_i^{\alpha,A}=\left\{
  \begin{array}{l}
    (X_i^\alpha)' \qquad \textrm{if} \quad i \in A\\
    X_i^\alpha \qquad \textrm{if} \quad i \notin A,
  \end{array}
  \right.
$$
namely the components of $X^\alpha$ with indices in $A$ are replaced by an independent copy. 
Let moreover $X^{A_\theta}=X^{A_\theta,*}=(X^{\theta,A},X^a)$, $X^{A_a}=(X^\theta,X^{a,A})$, $X^{A_a,*}=((X^{\theta})',X^{a,A})$. Let $f:S^n \to \Cm$ and introduce, for $j \in [n]$,
\bee
\Delta^\theta_j f(X)&=&f(X^\theta,X^a)-f(X^{\theta,j},X^a), \\
\Delta^a_j f(X)&=&f(X^\theta,X^a)-f(X^\theta,X^{a,j}).
\eee
We have the following lemma, adapted from Lemma 2.3 in \cite{chatterjee2008new}:

\begin{lemma} \label{lemCov}Let $|A|$ be the cardinal of $A \subseteq [n]$, and denote by $C^n_p$ the binomial coefficient. Then, 

 \be \label{cov}
      {\rm{Cov}}(h(X),f(X)) =\frac{1}{2}\sum_{\alpha=\theta,a}\sum_{A \subsetneq [n]} \frac{1}{C_{|A|}^n (n-|A|)} \sum_{j \notin A} \E \{ \Delta^\alpha_j h(X) \Delta^\alpha_j f(X^{A_\alpha,*}) \}.
      \ee
 Note that the empty set is included in the sum above over $A$.
\end{lemma}

\begin{proof}
  We have first
  $$
  {\rm{Cov}}(h(X),f(X))=\E\{h(X) f(X)\}-\E\{h(X)\}\E\{f(X)\}=\E\{h(X)(f(X)-f(X'))\},
  $$
  and we decompose then $f(X)-f(X')$ into two parts
      \bee
      f(X)-f(X')&=&f(X^\theta,X^a)-f((X^\theta)',X^a)\\
      &&+f((X^\theta)',X^a)-f((X^\theta)',(X^a)').
      \eee
      Each term is treated independently. 
      Proceeding as in \cite{chatterjee2008new}, we write
      \be \label{DF}
      f(X^\theta,X^a)-f((X^\theta)',X^a)=\sum_{A \subsetneq [n]} \frac{1}{C_{|A|}^n (n-|A|)} \sum_{j \notin A} \Delta^\theta_j f(X^{\theta,A},X^a),
      \ee
      and remark that
      \begin{align*}
        &\E \{h(X^\theta,X^a)(f(X^{\theta,A},X^a)-f(X^{\theta,A \cup j},X^a))\}\\
        &=\E \{h(X^{\theta,j},X^a)(f(X^{\theta,A\cup j},X^a)-f(X^{\theta,A},X^a))\}.
      \end{align*}
      As a consequence
       \begin{align*}
        &\E \{h(X^\theta,X^a)(f(X^{\theta,A},X^a)-f(X^{\theta,A \cup j},X^a))\}\\
        &=\frac{1}{2} \E \{\Delta^\theta_j h(X^{\theta},X^a) \Delta^\theta_j f(X^{\theta,A},X^a)\}.
      \end{align*}
Together with \fref{DF}, this gives the $\theta$ part of \fref{cov}. The other term follows similarly.
\end{proof}

\medskip

We will use the last lemma to evaluate $\E \{U_\eta(g)^2\}$ and $\E \{|U_\eta(g)|^2\}$. Note that there are other possible decompositions for the covariance in terms of $(\theta,\rho,a)$, and the one we used has the advantage of explicitly separating the contributions of $(\theta,\rho)$ and $a$. The set of integers $[n]$ in the lemma corresponds to an enumeration of the cell positions contained in $J_\eta(\omega)$, and a row of $X$ is $m_j(\omega)$. We recall that the $m_j(\omega)$ for different $j$ are independent and identically distributed. The set $J_\eta(\omega)$ is random because of the cells located at the boundary of $\calB$, and depending on $z(\omega)$, some boundary cells may or may not be part of $J_\eta(\omega)$. Cells sufficiently far from $\partial \calB$ are always in $J_\eta(\omega)$. Since $J_\eta(\omega)$ depends on $\omega$ only through $z$, we will denote in the sequel $J_\eta(\omega)=J_\eta(z)$.  

With $n:=|J^\eta(z)|$, $f(X)=\bar{h}(X)=U_\eta(g)$, and $u_\eta(g,\omega):=(g,u_\eta(\omega))$, we find, using Lemma \ref{lemCov}, after rearranging the sum, 
 \be \label{cov2}
    \E \{|U_\eta(g) |^2\}=\frac{1}{2 \eta^3}\sum_{j \in J^\eta(z)} \sum_{\alpha=\theta,a}\sum_{A \subset J^\eta(z)\backslash \{j\} } \frac{1}{C_{|A|}^n (n-|A|)} \E \{ \Delta^\alpha_j u_\eta(g,\omega) \Delta^\alpha_j \overline{u_\eta}(g,\omega^{A_\alpha,*}) \}.
    \ee
    Above, the notation $\omega^{A_\alpha,*}$ has the same meaning as that of $X^{A_\alpha,*}$, that is when $\alpha=\theta$, the $(\theta_j,\rho_j)$ with indices $j$ in $A$ are replaced by an independent copy. When $\alpha=a$, the definition is similar with in addition all $(\theta_j,\rho_j)_{j \in \Zm^3}$ replaced by an independent copy. The core of the analysis is then to obtain asymptotic expressions of $\Delta^\alpha_j u_\eta$, which involve two ingredients: (i) small volume expansions for $u_\eta(x,\omega)$ and for related Green's functions; they arise since only the inclusion $j$ is modified in $\Delta^\alpha_j u_\eta$, yielding a contribution of order of its volume, namely $O(\eta^3)$, and (ii) two-scale expansions for $u_\eta(x,\omega)$ and for other functions, needed to obtain the leading term in $\Delta^\alpha_j u_\eta$. The two-scale expansion \fref{2sca} holds in the interior of $\calB$ and not close to its boundary due to boundary effects. In the set $J^\eta(z)$, we will therefore consider only those indices $j$ that are associated with cells located at least at a distance $\delta>0$ from the boundary, $\delta$ independent of $\eta$. This set is denoted by $J^\eta_\delta$. 
The number of cells at a distance less than $\delta$ is of order $O(\delta \eta^{-3})$, and since $\Delta^\alpha_j u_\eta$ is of order $O(\eta^3)$ as claimed earlier, discarding cells at the boundary introduces an error of order $O(\delta)$ in \fref{cov2}. 
We have then
\bea \nonumber
\E\{|U_\eta(g) |^2\}&=&\frac{1}{2 \eta^3 }\sum_{j \in J^\eta_\delta} \sum_{\alpha=\theta,a}\sum_{A \subset J^\eta(z) \backslash \{j\}} \frac{1}{C_{|A|}^n (n-|A|)} \E \{ \Delta^\alpha_j u_\eta(g,\omega) \Delta^\alpha_j \overline{u_\eta}(g,\omega^{A_\alpha,*}) \}\\
&&+O(\delta), \label{cov3}
    \eea
with a similar expression for $\E \{U_\eta(g)^2\}$. The boundary effects are then of order $O(\delta)$ and are negligible.
      
Starting from \fref{cov3}, the next step is to investigate $u_\eta(g,\omega)$ and derive asymptotic expressions. 

      \section{Small-volume expansions} \label{secSV}

      \subsection{Preliminaries} We denote by $G_\eta(x,x',\omega)$ the Green's function satisfying
      $$
      \nabla_x \cdot a_\eta  \nabla_x G_\eta + k_0^2 G_\eta=\delta(x-x'), \qquad x,x' \in \Rm^3,
$$
      equipped with radiation conditions. Starting from $a_\eta$, we remove the inclusion located at cell $j$ and denote the resulting coefficient by $a_\eta^{j,-}$. The associated Green's function is $G_\eta^{j,-}$. 
      We then write
       \bee
       \nabla_x \cdot a_\eta(\omega)  \nabla_x u_\eta(\omega)&=&\nabla_x \cdot (a_\eta^{j-}+(\eta^2 a_j(\omega)-a_b)\un_{D_\eta^j(\omega)}(x))  \nabla_x u_\eta(\omega)\\
       \nabla_x \cdot a_\eta(\omega^{j_\alpha})  \nabla_x u_\eta(\omega^{j_\alpha})&=&\nabla_x \cdot (a_\eta^{j-}+(\eta^2 a_j(\omega^{j_\alpha})-a_b)\un_{D_\eta^j(\omega^{j_\alpha})}(x))  \nabla_x u_\eta(\omega^{j_\alpha}),
       \eee
       where we recall that in $\omega^{j_\alpha}$, $(\theta_j,\rho_j)$ is replaced by an independent copy when $\alpha=\theta$, and that $a_j$ is replaced by an independent copy when $\alpha=a$. Injecting the relations above in the respective Helmholtz equations, taking the difference, multiplying by $G_\eta^{j,-}(x,x',\omega)$, and integrating by parts give
        \bea \label{diffu}
        u_\eta(x,\omega)-u_\eta(x,\omega^{j_\alpha}) &=&  \int_{D_\eta^j(\omega)} (\eta^2 a_j(\omega)-a_b) \nabla_{x'} G^{j-}_\eta(x,x',\omega) \cdot \nabla u_\eta(x',\omega) dx'\\\nonumber
        &&-\int_{D_\eta^j(\omega^{j_\alpha})} (\eta^2 a_j(\omega^{j_\alpha})-a_b) \nabla_{x'} G^{j-}_\eta(x,x',\omega) \cdot \nabla u_\eta(x',\omega^{j_\alpha}) dx'. 
        \eea
        To move forward, we will need the results of the next paragraph.
        
          \paragraph{Two-scale expansions for $G_\eta^{j,-}$.} We recall $G_\h$ is the homogenized Green's function solution to \fref{homoGF}. For $j \in \Zm^3$,  let $\Lambda^{j,-}$ be equal to $\Lambda$  except that $\Lambda^{j,-}=1$ on $D^j(\omega)$. We show in Appendix \ref{2scale} that the two-scale expansion of $G_\eta^{j,-}$ is
      \be \label{2sGm}
      G_\eta^{j,-}(x,x',\omega)= \Lambda^{j,-}(x/\eta,\omega) G_{\h}(x,x')+\eta \phi^{j,-}(x/\eta,\omega) \cdot \nabla_x G_{\h}(x,x')+O(\eta^2),
      \ee
      where $\phi^{j,-}$ was introduced in Section \ref{setting} and solves \fref{corr1m}-\fref{corr2m}-\fref{corr3m}. Going back to \fref{diffu}, and with the notation $G^{j,-}_{\eta,g}(x',\omega)=\int_{\Rm^3} G^{j-}_\eta(x,x',\omega) g(x) dx$, we define
         \be \label{defG}
      \calG_j(\omega)=\int_{D_\eta^j(\omega)} (\eta^2 a_j(\omega)-a_b) \nabla G^{j,-}_{\eta,g}(x',\omega) \cdot \nabla u_\eta(x',\omega) dx',
      \ee
      so that
      $$
      \Delta^\alpha_j u_{\eta}(g,\omega)=\calG_j(\omega)-\calG_j(\omega^{j_\alpha}).
$$
      Note that in the latter expression we have used the fact that $G^{j,-}_{\eta,g}(x',\omega)=G^{j,-}_{\eta,g}(x',\omega^{j_\alpha})$ since $G_\eta^{j,-}$ is independent of the $j$ inclusion.
      
    \subsection{Expansions of $\Delta^\alpha_j u_{\eta}$}
      We now exploit the two-scale expansions of $u_\eta$ and $G_\eta^{j,-}$ to obtain the leading terms in $\Delta^\alpha_j u_{\eta}$. 
      For $I$ the $3\times 3$ identity matrix, we will use the notations
      \bea \label{defAB}
      Q(y,\omega)&=&I+(\nabla \phi_1(y,\omega), \cdots,\nabla  \phi_3(y,\omega))\\
      Q^{j,-}(y,\omega)&=&I+(\nabla \phi^{j,-}_1(y,\omega), \cdots,\nabla  \phi^{j-}_3(y,\omega))\label{2sGum}
    \eea
    as well as $G_{\rm{h},g}(x')=\int_{\Rm^3} G_{\rm{h}}(x',x)g(x) dx'$. The two-scale expansions of $u_\eta$ and $G^{j-}_{\eta}$ yield
    \bea
 \label{2sGu}   \nabla u_\eta (x,\omega)&=&\eta^{-1} (\nabla \Lambda)(x/\eta,\omega) u_\h(x)+Q(x/\eta,\omega) \nabla u_\h(x)+O(\eta)\\
  \label{2sGG}  \nabla G^{j,-}_{\eta,g} (x,\omega)&=&\eta^{-1} (\nabla \Lambda^{j,-})(x/\eta,\omega) G_{\h,g}(x) +Q^{j,-}(x/\eta,\omega)\nabla G_{\rm{h},g}(x)+O(\eta).
  \eea
  Note that the function $\phi$ has not been defined on $D^\ell(\eta)$, $\ell \in \Zm^3$, which is not a problem since it will not be needed. The terms $Q(x/\eta,\omega)$ and $Q^{j,-}(x/\eta,\omega)$ above have then to be only considered outside of $D^\ell(\eta)$ for all $\ell \in \Zm^3$ for $Q$ and outside of $D^\ell(\eta)$ for all $\ell \neq j$ for $Q^{j,-}$. 
  
    When $x \in D^j_\eta(\omega)$, the term of order $\eta^{-1}$ in \fref{2sGG} is equal to zero, but not the one in \fref{2sGu}. This seems to imply, naively, that since the volume of $D_\eta^j(\omega)$ is of order $\eta^3$, the term $\calG_j(\omega)$ is of order $\eta^2$. We will actually see after an integration by parts that the term of order $\eta^{-1}$ in $\nabla u_\eta (x,\omega)$ yields a vanishing contribution to $\calG_j(\omega)$, and as a consequence $\calG_j(\omega)$ is of order  $\eta^3$ and not  $\eta^{2}$.

      Instead of directly injecting \fref{2sGu} and \fref{2sGG} into \fref{defG} and handling a posteriori $\nabla \Lambda$, we first transform $\calG_j$ in order to identify the leading terms: we find, after an integration by parts,
      \bee
      \calG_j(\omega)&=&\int_{\partial D_\eta^j(\omega)} (\eta^2 a_j(\omega)-a_b) n^j \cdot \nabla G^{j,-}_{\eta,g}(x',\omega)  u_\eta(x',\omega) dS(x')\\
      &&-\int_{ D_\eta^j(\omega)} (\eta^2 a_j(\omega)-a_b) \Delta G^{j,-}_{\eta,g}(x',\omega) u_\eta(x',\omega) dx'\\
      &=&\int_{\partial D_\eta^j(\omega)} (\eta^2 a_j(\omega)-a_b) n^j \cdot \nabla G^{j,-}_{\eta,g}(x',\omega)  u_\eta(x',\omega) dS(x')\\
 &&+a_b^{-1} k_0^2\int_{ D_\eta^j(\omega)} (\eta^2 a_j(\omega)-a_b)  G^{j,-}_{\eta,g}(x',\omega)  u_\eta(x',\omega) dx'\\
      &:=& \calG_j^0(\omega)+\calG_j^1(\omega).
      \eee
      In the second equality above, we used the fact that $G^{j,-}_{\eta,g}$ verifies the Helmholtz equation with no inclusion at cell $j$. We recall that $n^j$ is the outward normal to $\partial D_\eta^j(\omega)$ and that $dS$ is the surface measure on $\partial D_\eta^j(\omega)$.
      
      We have the following lemma that provides us with an asymptotic expansion of $\calG_j(\omega)$:
      \begin{lemma} \label{LemM} Let $j \in J_\eta^\delta$. With
        $$
        M^j_{\ell,k}(\omega)=\int_{\partial D^j(\omega)} n^j \cdot (\nabla y_\ell + \nabla \phi^{j,-}_\ell(y,\omega)) \left(y_k+\phi_k (y,\omega) \right) dS(y)
        $$
        for $\ell,k=1,2,3$ and
        $$
        N^j(\omega)=\int_{D^j(\omega)} (1+\Lambda(y,\omega))dy,
        $$
        we have,
        $$
        \calG_j(\omega)=-a_b \eta^3 \nabla G_{\rm{h},g}(j \eta) \cdot M^j(\omega) \nabla u_{\rm{h}}(j\eta )+k_0^2 \eta^3 N^j(\omega) G_{\rm{h},g}(j \eta) u_{\rm{h}}(j \eta) +O(\eta^4).
        $$
        \end{lemma}
      \begin{proof}
      We start by plugging the two-scale expansion $u_\eta(x,\omega)=u_{\rm{h}}(x) \Lambda(x/\eta,\omega)+\eta \phi(x/\eta,\omega)\cdot \nabla u_{\rm{h}}(x)+O(\eta^2)$ into $\calG_j^0$, and denote the respective terms by $\calG_j^{00}$ and $\calG_j^{01}$. Since the surface of $D_\eta^j(\omega)$ is of order $\eta^2$, it may appear as if $\calG_j^0$ is of order $\eta^2$. We will see after an integration by parts that $\calG_j^0$ is actually of order $\eta^3$. We find indeed, using that $\Lambda=1$ on $\partial D_\eta^j$, 
        \begin{align} \nonumber
    \calG_j^{00}(\omega)&:=\int_{\partial D_\eta^j(\omega)} (\eta^2 a_j(\omega)-a_b) n^j \cdot \nabla G^{j,-}_{\eta,g}(x',\omega)  u_{\rm{h}}(x') dS(x')\\ 
    &\nonumber = \int_{D_\eta^j(\omega)} (\eta^2 a_j(\omega)-a_b) \nabla G^{j,-}_{\eta,g}(x',\omega)\cdot \nabla  u_{\rm{h}}(x') dx'\\
                     & \label{G00} \qquad -a_b^{-1}k_0^2\int_{ D_\eta^j(\omega)} (\eta^2 a_j(\omega)-a_b) G^{j,-}_{\eta,g}(x')  u_{\rm{h}}(x') dx',
    \end{align}
    where we used once more the fact that $G^{j,-}_{\eta,g}$ verifies the Helmholtz equation with no inclusion at cell $j$. Injecting \fref{2sGG} in the first term in \fref{G00}, and writing $\nabla u_{\rm{h}}(x')=\nabla u_{\rm{h}}(j \eta)+O(\eta)$ as well as $\nabla G_{\rm{h},g}(x')=\nabla G_{\rm{h},g}(j \eta)+O(\eta)$  for $x' \in D^j_\eta(\omega)$, we have
        \begin{align} \nonumber
 &\int_{D_\eta^j(\omega)} (\eta^2 a_j(\omega)-a_b) \nabla G^{j,-}_{\eta,g}(x',\omega)\cdot \nabla  u_{\rm{h}}(x') dx'\\ 
 &=  -a_b     \left[\left(\int_{D_\eta^j(\omega)} Q^{j,-}(x'/\eta,\omega) dx'\right) \nabla  G_{\rm{h},g}(j \eta ) \right]\cdot \nabla  u_{\rm{h}}(j \eta) +O(\eta^4).\label{firstG0}
        \end{align}
        The term in parentheses above can be recast as 
        \bee \int_{D_\eta^j(\omega)} \left(Q^{j,-}(x'/\eta,\omega)\right)_{k,\ell} dx'&=&\int_{D_\eta^j(\omega)} (\delta_{k \ell} + \partial_{y_k} \phi^{j,-}_\ell (x'/\eta,\omega)) dx'\\
        &=&\eta^3 \int_{D^j(\omega)} (\delta_{k \ell} + \partial_{y_k} \phi^{j,-}_\ell (y,\omega)) dy\\
                       &=&\eta^3 \int_{D^j(\omega)} \nabla y_k \cdot (\nabla y_\ell+ \nabla  \phi^{j,-}_\ell(y,\omega)) dy\\
      &=&\eta^3 \int_{\partial D^j(\omega)} y_k\, n^j \cdot (\nabla y_\ell+ \nabla  \phi^{j,-}_\ell(y,\omega)) dS(y).
        \eee
    In the last equality, we used that $\Delta  \phi^{j,-}_\ell=0$ in $D^j(\omega)$. For the second term in \fref{G00}, we find, using the two-scale expansion of $G^{j,-}$ given in \fref{2sGm}:
      \begin{align} \nonumber
  -&a_b^{-1}k_0^2\int_{ D_\eta^j(\omega)} (\eta^2 a_j(\omega)-a_b) G^{j,-}_{\eta,g}(x')  u_{\rm{h}}(x') dx'\\
       &= k_0^2 |D_\eta^j(\omega)| G_{\rm{h},g}(j \eta) u_{\rm{h}}(j \eta) +O(\eta^4). \label{secondG0}
      \end{align}
      Collecting \fref{firstG0} and \fref{secondG0} then gives the leading terms in $\calG^{00}_j$. We now turn to $\calG_j^{01}$, and find
         \begin{align*} \nonumber
    \calG_j^{01}(\omega)&:=\eta \int_{\partial D_\eta^j(\omega)} (\eta^2 a_j(\omega)-a_b) \big[ n^j \cdot \nabla G^{j,-}_{\eta,g}(x',\omega) \big] \big[ \phi(x/\eta',\omega)\cdot \nabla u_{\rm{h}}(x') \big] dS(x')\\ \nonumber
                        &\nonumber = -\eta a_b  \left[\int_{\partial D_\eta^j(\omega)} \big[n^j \cdot \left(Q^{j,-}(x'/\eta,\omega) \nabla  G_{\rm{h},g}(j \eta) \right)  \big] \phi(x/\eta',\omega)  dS(x') \right]\cdot \nabla u_{\rm{h}}(j\eta )+O(\eta^4)\\
                        & \nonumber = -\eta a_b \sum_{i,\ell, k=1}^3 \partial_{x_\ell} G_{\rm{h},g}(j \eta) \left[\int_{\partial D_\eta^j(\omega)} n^j_i\, Q_{i\ell}^{j,-}(x'/\eta,\omega)\, \phi_k(x'/\eta,\omega)\,  dS(x') \right]\partial_{x_k} u_{\rm{h}}(j\eta )+O(\eta^4)\\
                        &\nonumber  = -\eta^3 a_b \sum_{\ell, k=1}^3 \partial_{x_\ell} G_{\rm{h},g}(j \eta) \left[\int_{\partial D^j(\omega)} n^j \cdot (\nabla y_\ell + \nabla \phi^{j,-}_\ell)(y,\omega) \phi_k (y,\omega)  dS(y) \right] \partial_{x_k} u_{\rm{h}}(j\eta )\\
           &\qquad +O(\eta^4).
         \end{align*}
         Gathering this and the expansion of $\calG_j^{00}$ finally gives
       $$
    \calG_j^0(\omega)=-a_b \eta^3 \nabla G_{\rm{h},g}(j \eta) \cdot M^j(\omega) \nabla u_{\rm{h}}(j\eta )+k_0^2 \eta^3 |D^j(\omega)| G_{\rm{h},g}(j \eta) u_{\rm{h}}(j \eta) +O(\eta^4).
    $$
To conclude, it is enough to plug the two-scale expansions of $ G^{j,-}_{\eta,g}$ and $u_\eta$ into $\calG_j^{1}$ to obtain
    $$\calG_j^1(\omega)= k_0^2 \eta^3 \int_{D^j(\omega)} \Lambda(y,\omega) dy \; G_{\rm{h},g}(j \eta) u_{\rm{h}}(j \eta)+O(\eta^4).
    $$
This ends the proof
\end{proof}

Owing to the preceding lemma, we are now in position to state expansions of $\Delta_j^\theta u_{\eta}(g,\omega)$ and $\Delta_j^a u_{\eta}(g,\omega)$:

    \begin{corollary} Let $j \in J_\eta^\delta$. Then, 
      \bea
 \label{exp1}      \Delta_j^\theta u_{\eta}(g,\omega)&=&-\eta^3 a_b  \nabla  G_{\rm{h},g}(j \eta )\cdot \big(  M^j(\omega)-M^j(\omega^{j_\theta})\big) \nabla  u_{\rm{h}}(j \eta )\\ \nonumber
       &&+ \eta^3 k_0^2  G_{\rm{h},g}( j\eta ) \big(N^j(\omega)-N^j(\omega^{j_\theta})\big) u_{\rm{h}}(j \eta )+O(\eta^{4})\\\label{exp2}
         \Delta_j^a u_{\eta}(g,\omega)
      &=&  \eta^3 k_0^2 G_{\rm{h},g}(j \eta )  \big(N^j(\omega)-N^j(\omega^{j_a})\big) u_{\rm{h}}(j \eta )+O(\eta^{4}).
      \eea
      \end{corollary}
The proof follows readily from Lemma \ref{LemM} and the observation that $M^j(\omega)=M^j(\omega^{j_a})$ since $\phi$ and $\phi^{j,-}$ do not depend on $a$.

 The next step is to evaluate $\E \{ \Delta^\alpha_j u_\eta(g,\omega) \Delta^\alpha_j \overline{u_\eta}(g,\omega^{A_\alpha,*}) \}$ for $\alpha=\theta,a$ using the results of this section. This involves estimating various terms of the form $\E \{ M^j_{i,\ell}(\omega) M^j_{p,q}(\omega^{A_\alpha,*})\}$ for $A \subset J^\eta(z) \backslash \{j\}$, which requires an explicit use of the probabilistic framework of Section \ref{setting}.

    \section{The asymptotic second order moments} \label{secvar}

    
Before investigating the variance itself, we need to establish some statistical properties of the modified corrector $\phi^{j,-}$. We already know that $\nabla \phi_i$ is stationary, which we recall is equivalent to $\nabla \phi_i(x+y,\omega)= \nabla \phi_i(y,T_x \omega)$ for all $x \in \Rm^3$. As already mentioned, this is not true for $\nabla \phi^-$ since there is no inclusion in the cell $j$ and stationarity cannot be expected. Yet, we will see that there is a comparable relation for $\phi^{j,-}$ that is enough for our analysis.

To see this, given $\omega \in \widetilde{\Omega}$, let $(\omega)_0=((m_{k}(\omega))_{k \in \Zm^3_\star},z(\omega))$, namely we remove the $0$ component of $m_{k}(\omega)$ in $\omega$, and introduce
\bee
      \calD_0((\omega)_0) &=&\bigcup_{k \in \Zm^3,\, k \neq 0}B(\theta_{k}(\omega)+k-z(\omega),\rho_{k}(\omega)).
      \eee
We have then the following lemma:

\begin{lemma} \label{stat}For $i=1,2,3$ and $j \in \Zm^3$, there exists $\phi_i^-(y,(\omega)_0)$, unique up to an additive constant, such that $\nabla \phi_i^{j,-}(x,\omega)=\nabla \phi_i^-(x-j,(T_j \omega)_0)$, and such that $\Delta \phi^-_i(y,(\omega)_0)=0$ on $\calD_0((\omega)_0)^c $, with the boundary conditions
      $$
      n \cdot \nabla \phi^-_i=-n \cdot e_i \qquad \textrm{on} \qquad \partial \calD_0((\omega)_0),
      $$
      and  the condition at infinity $\lim_{|y| \to \infty} \un_{\calD_0((\omega)_0)^c}(y) (\nabla \phi_i^{-}(y,\omega)-\nabla \phi_i(y,\omega))=0$.
  \end{lemma}

  \begin{proof} The idea is simply to perform shifts to translate the $j$ cell to the origin. We start by rewriting $\calD_j(\omega)$ appropriately, and recall that
    $$
    \calD_j(\omega)=\bigcup_{k \in \Zm^3, k \neq j}D^k(\omega)= \bigcup_{k \in \Zm^3,\, k \neq j} B(\theta_k(\omega)+k -z(\omega),\rho_k(\omega)),
    $$
   We show first that $ \calD_j(\omega)=j+\calD_0(T_j\omega)$, with $\calD_0(\omega)=\calD_{j=0}(\omega)$. The notation $j+\calD_0(T_j\omega)$ means that the centers of the balls in $\calD_0(T_j\omega)$ are all shifted by $j$. We have indeed, since $[j+z(\omega)]=j$ and $z(T_j \omega)=z(\omega)$ when $j \in \Zm^3$ and $z(\omega) \in Z$,
      \bee
      \calD_j(\omega) 
      &=&\bigcup_{k \in \Zm^3,\, k \neq j} B(\theta_k(\omega)+j+k-j-z(\omega),\rho_k(\omega))\\
      &=&\bigcup_{k \in \Zm^3,\, k \neq 0}B(\theta_{k+[j+z(\omega)]}(\omega)+ j+k+[j+z(\omega)]-j-z(\omega),\rho_{k+[j+z(\omega)]}(\omega))\\
      &=&\bigcup_{k \in \Zm^3,\, k \neq 0}B(\theta_{k}(T_j\omega)+j+k-z(T_j\omega),\rho_{k}(T_j\omega))\\
      &=&j+\calD_0(T_j\omega).
      \eee
      Above, we used that 
       $(\theta_{k+[j+z(\omega)]}(\omega),\rho_{k+[j+z(\omega)]}(\omega))=(\theta_{k}(T_j\omega),\rho_{k}(T_j\omega)).$
      Note that $\calD_0(T_j\omega)$ does not depend on the zero component of $T_j \omega$ and we recast it as $\calD_0(T_j\omega)=\calD_0((T_j\omega)_0)$. For $i=1,2,3$, let now $\phi_i^-(y,(T_j\omega)_0)$ be the solution to $\Delta \phi^-_i=0$ on $\calD_0((T_j\omega)_0)^c $, with the boundary conditions
      $$
      n \cdot \nabla \phi^-_i=-n \cdot e_i \qquad \textrm{on} \qquad \partial \calD_0((T_j\omega)_0),
      $$
      and the condition at infinity $\lim_{|y| \to \infty} \un_{\calD_0((T_j \omega)_0)^c}(y) (\nabla \phi_i^{-}(y,(T_j \omega)_0))-\nabla \phi_i(y,T_j \omega))=0$.
      The function $\phi^-$ is built in the same manner as $\phi^{j,-}$ (see Appendix \ref{phijm}). 
      It is then clear that the shifted function $\phi_i^-(x-j,(T_j \omega)_0)$ satisfies \fref{corr1m}-\fref{corr2m}-\fref{corr3m}, which admits a unique solution up to an additive constant. As a consequence $\nabla \phi_i^-(x-j,(T_j \omega)_0)=\nabla\phi_i^{j,-}(x,\omega)$, which proves the claim.
      \end{proof}

      With the previous lemma at hand, we can now express $M^j(\omega)$ and $N^j(\omega)$ in terms of $M$ and $N$ defined in \fref{defMF}-\fref{defNF}. We need first to introduce some notation. Let $m=(m_\ell)_{\ell \in \Zm^3}$. Then $\tau_j m:=(m_{\ell+j})_{\ell \in \Zm^3}$ for $j \in \Zm^3$. For a subset $A$ of $J^\eta(z) \backslash \{j\}$, the notation $m^{A_\alpha-j,*}$ for $\alpha=\theta $ means that the component $(\theta_k,\rho_k)$ of $m$ for $k=\ell-j$ with $\ell \in A$ is replaced by an independent copy $(\theta_k',\rho_k')$, while when $\alpha=a$, $a_k$ is replaced by $a_k'$ and the entire $(\theta_\ell,\rho_\ell)_{\ell \in \Zm^3}$ is replaced by an independent copy. We have then:

      \begin{corollary} \label{corM} With the notations above, and $\omega=(m,z)$:
        \begin{align*}
          &M^j(\omega)=M(\tau_j m), \qquad  M^j(\omega^{A_\alpha,*})=M(\tau_j m^{A_\alpha-j,*})\\
          &N^j(\omega)=N(\tau_j m), \qquad N^j(\omega^{A_\alpha,*})=N(\tau_j m^{A_\alpha-j,*}).
        \end{align*}
        \end{corollary}

        \begin{proof}
         We start with $M^j(\omega)$. We have, after writing $D^j(\omega)=j-z(\omega)+B((\theta_j(\omega),\rho_j(\omega)))$, and using Lemma \ref{stat} as well as the stationarity of $\nabla \phi$ in the third equality,
      \bee
      M^j(\omega)&=&\int_{\partial D^j(\omega)} n^j \cdot (\nabla y_\ell + \nabla \phi^{j,-}_\ell(y,\omega)) \left(y_k+\phi_k (y,\omega) \right) dS(y)\\
      &=&\int_{\partial B(\theta_j(\omega),\rho_j(\omega))} n^j \cdot (\nabla y_\ell + \nabla \phi^{j,-}_\ell(y+j-z(\omega),\omega)) \\
      &&\qquad \times \left(y_k+(j-z(\omega))_k+\phi_k (y+j-z(\omega),\omega) \right) dS(y)\\
      &=&\int_{\partial B(\theta_j(\omega),\rho_j(\omega))} n^j \cdot (\nabla y_\ell + \nabla \phi^{-}_\ell(y,(T_{j-z(\omega)}\omega)_0) \\
      &&\qquad \times \left(y_k+(j-z(\omega))_k+\phi_k (y,T_{j-z(\omega)}\omega)+C_{j-z(\omega)}(\omega) \right) dS(y).
        \eee
        Besides, since $j \in \Zm^3$,
        $$
T_{j-z(\omega)} \omega= T_{j-z(\omega)} \big((m_k)_{k \in \Zm^3},z\big)=\big((m_{k+j})_{k \in \Zm^3},0\big)=(\tau_j m,0),
$$
and also $(\theta_{j}(\omega),\rho_{j}(\omega))=(\theta_{[j+z(\omega)]}(\omega),\rho_{[j+z(\omega)]}(\omega))=(\theta_{0}(T_j\omega),\rho_{0}(T_j\omega))$. Since $(\theta_0(\omega),\rho_0(\omega))$ does not depend on the $z$ component of $\omega$, we then write
$(\theta_{j}(\omega),\rho_{j}(\omega))=(\theta_{0}(\tau_j m),\rho_{0}(\tau_j m))$. In the same way, we write 
$\phi_k (y,\tau_j m)$ for $\phi_k (y,T_{j-z(\omega)}\omega)$.

Now, since  $\Delta \phi^{j,-}_\ell(y,\omega)=0$ on $D^j(\omega)$, it follows that
$$
\int_{\partial B(\theta_j(\omega),\rho_j(\omega))} n^j \cdot (\nabla y_\ell + \nabla \phi^{-}_\ell(y,(T_{j-z(\omega)}\omega))_0) dS(y)=0.
$$
As a consequence, the term proportional to $(j-z(\omega))_k+C_{j-z(\omega)}(\omega)$ in $M^j(\omega)$ vanishes and 
 \bee
 M^j(\omega)&=&\int_{\partial B(\theta_0(\tau_j m),\rho_0(\tau_j m))} n^j \cdot (\nabla y_\ell + \nabla \phi^{-}_\ell(y,(\tau_j m)_0)  \left(y_k+\phi_k (y,\tau_j m)\right) dS(y)\\
 &=&M(\tau_j m),
 \eee
 which proves the claim. Note that this also shows that $M$ is uniquely defined as adding a constant to $\phi_k$ produces a vanishing term.  
        When considering $M^j(\omega^{A_\alpha,*})$, some bookkeeping is necessary to keep track of the indices that are in $A$ and a similar calculation as above shows that
        $$
        M^j(\omega^{A_\alpha,*})=M(\tau_j m^{A_\alpha-j,*}).
        $$
        
             Similarly, since the function $\Lambda(x,\omega)$ is stationary, i.e. verifies $\Lambda(x,\omega)=\Lambda(x-j,T_j\omega)$, we find
             \bee
             N^j(\omega)&=&\int_{D^j(\omega)} (1+\Lambda(y,\omega))dy\\
             &=&\int_{B(\theta_j(\omega),\rho_j(\omega))} (1+\Lambda(x,T_{j-z(\omega)}\omega)) dx=N(\tau_jm).
             \eee
            The relation   $
        N^j(\omega^{A_\alpha,*})=N(\tau_j m^{A_\alpha-j,*})
        $ follows in the same manner. This ends the proof.
      \end{proof}
      
\medskip 
We can finally now start the analysis of $\E\{|U_\eta(g)|^2\}$. We begin with \fref{cov3}, and inject the expansions \fref{exp1}-\fref{exp2}. There are various terms to study, and we focus  on the following one: for $j \in J_\eta^\delta$, let
\bee 
\calC^{j,\eta}&=&\frac{a_b^2}{2}\sum_{A \subset J^\eta(z) \backslash \{j\}} \frac{1}{C_{|A|}^n (n-|A|)} \\
&& \qquad \times \E\big \{\left(M^j(\omega)-M^j(\omega^{j_\theta})\right)\otimes \left(M^j(\omega^{A_\theta,*})-M^j(\omega^{(A\cup j)_\theta,*})\right)\big \}.
\eee

Our goal is to show that $\calC^{j,\eta}$ is asymptotically independent of $j$ and $z$ and to identify the limit. A first immediate step towards this is to use Corollary \ref{corM}, which gives, for $A \subset J^\eta(z) \backslash \{j\}$ and $j \in J^\delta_\eta$, 
        \begin{align*}
          &\E\big \{\left(M^j(\omega)-M^j(\omega^{j_\theta})\right)\otimes \left(M^j(\omega^{A_\theta,*})-M^j(\omega^{(A\cup j)_\theta,*})\right)\big \}\\
          &= \E\big\{\left(M(\tau_j m)-M(\tau_j m^{j_\theta-j})\right) \otimes \left(M( \tau_j m ^{A_\theta-j,*})-M(\tau_j m^{(A\cup j)_\theta-j,*})\right)\big\}   \\
          &= \E\big\{\left(M(m)-M(m^{0_\theta})\right) \otimes \left(M( m ^{U_\theta,*})-M(m^{(U \cup 0)_\theta,*})\right)\big\}
        \end{align*}
        where $U \subset J_{\eta,j}(z)=\{\ell \in \Zm^3: \ell+j\in J^\eta(z) \backslash{\{j\}}\}$ (to be more precise, as $A$ runs through $J^\eta(z) \backslash \{j\}$, then $U$ runs through $J_{\eta,j}(z)$, which is all we need in the sequel). Above, we used the fact that the $(\theta_k,\rho_k,a_k)$ for $k \in \Zm^3$ are identically distributed to remove the translation operator $\tau_j$.

        For $A \subset J_{\eta,j}(z)$, let now
        $$
        F(A)=\E\big\{\left(M(m)-M(m^{0_\theta})\right) \otimes \left(M( m ^{A_\theta,*})-M(m^{(A \cup 0)_\theta,*})\right)\big\},
        $$
        so that $\calC^{ j,\eta}$ reads 
        \bee 
\calC^{j,\eta}=\frac{a_b^2}{2}\sum_{A \subset J_{\eta,j}(z)} \frac{1}{C_{|A|}^n (n-|A|)} F(A).
\eee
The term $\calC^{j,\eta}$ depends on $j$ and $z$ via the set $J_{\eta,j}(z)$. We show in the next proposition that, as $\eta \to 0$, the set $J_{\eta,j}(z)$ can be replaced by a set containing the indices of the cells that are included in a ball of radius of order $\eta^{-1}$ centered at the origin. That set does not depend on $j$, and these cells are away from the boundary $\calB$, which will take care of the dependency on $z$. The key part in the analysis is to show that inclusions that are sufficiently far away from the origin do not contribute to $\calC^{j,\eta}$ in average as $\eta \to 0$. 

\begin{proposition} For the $\delta>0$ introduced in Section \ref{setting}, let $I_\eta=\{ \ell \in \Zm^3,\, \ell  \neq 0: \ell+Z \subset B_{\delta \eta^{-1}/2}\}$, and denote by $N_\eta$ its cardinal. Then,
 $$ \calC^{j,\eta}=\frac{a_b^2}{2}\sum_{A \subset I_\eta} \frac{1}{C_{|A|}^{N_\eta} (N_\eta-|A|)} F(A)+O(\eta).
  $$
\end{proposition}

\begin{proof} We remark first that $j+I_\eta \subset J_\eta(z)$ when $j \in J_\eta^\delta$ since $j$ is at least at a distance $\delta \eta^{-1}$ to the boundary of the domain.

\paragraph{Step 1: decomposing $J_{\eta,j}(z)$.} We split now $J_{\eta,j}(z)$ as $J_{\eta,j}(z)=I_\eta \cup I_\eta^c$ where $I_\eta^c$ is the complementary of $I_\eta$ in $J_{\eta,j}(z)$. Note that $I_\eta$ does not depend on $j$ while $I_\eta^c$ does. Recall that $n=|J_\eta(z)|$ and $N_\eta=|I_\eta|$. Since $|J_{\eta,j}(z)|=n-1$, it is clear that $|I_\eta|=n-1-N_\eta$. We have then
\bee
\frac{2}{a_b^2}\calC^{j,\eta}&=&\sum_{A \subset J_{\eta,j}(z)} \frac{1}{C_{|A|}^n (n-|A|)} F(A)=\frac{1}{n}\sum_{A \subset J_{\eta,j}(z)} \frac{1}{C_{|A|}^{n-1}} F(A)\\
&=&\frac{1}{n}\sum_{
\tiny{\begin{array}{l}
  A=A_1 \cup A_2\\
  A_1 \subset I_\eta,\; A_2 \subset I_\eta^c
\end{array}}
} \frac{1}{C_{|A|}^{n-1}} F(A_1 \cup A_2)\\
&=&\frac{1}{n}\sum_{m=0}^{n-1}\sum_{
\tiny{\begin{array}{l}
  A=A_1 \cup A_2, \; |A|=m\\
  A_1 \subset I_\eta,\; A_2 \subset I_\eta^c
\end{array}}
} \frac{1}{C_{m}^{n-1}} F(A_1 \cup A_2)
\eee

Consider $A_2 \subset I_\eta^c$ such that $|A_2|=p \geq 0$. We now remove the indices in $A_2$ one by one to attain the empty set. Let for this $\ell_i \in \Zm^3$, $i=1,\cdots,p$ be the indices in $A_2$.  We then write
\bee
F(A_1 \cup A_2)&=&F_(A_1 \cup A_2)-F( A_1 \cup (A_2\backslash \{\ell_1 \}))\\
&&+F( A_1 \cup (A_2\backslash \{\ell_1 \}))-F( A_1 \cup (A_2\backslash \{\ell_1,\ell_2 \}))+\cdots+F( A_1).
\eee
We need to show that, when removing an index $\ell_i$ in $A_2$, that is when replacing the component $m_{\ell_i}$ by another realization of the inclusion parameters, the value of $\calC^{\theta, j,\eta}$ as $\eta \to 0$ does not change. This requires us to investigate how sensitive are the corrector and modified corrector to changes in the inclusions parameters. We focus on $F(A_1 \cup A_2)-F( A_1 \cup (A_2\backslash \{\ell_1 \}))$ since the other terms are treated in the same manner.

\paragraph{Step 2: estimates on correctors.} We use the representation formula \fref{repr} obtained in Appendix \ref{phijm}:
\begin{align*} \nonumber
         &\phi_i(x,\omega)-\phi_i^{j,-}(x,\omega)\\\nonumber
                   &\qquad =\int_{ \partial D^j(\omega)} n^j_i \, G_j(x,y) dS+
           \int_{ \partial D^j(\omega)} \phi_i(y,\omega) \, n^j \cdot \nabla_y G_j(x,y) dS,
\end{align*}
where $G_j$ is the Green's function defined in Appendix \ref{phijm}. Given that the decay of $G_j$ and its derivatives are estimated in Appendix \ref{secG}, the latter formula gives us a way to quantify how sensitive is $\phi_i(x,\omega)$ to the inclusion $j$ when $x$ is large. In what follows, we recall that $\omega=(m,0)$, and we will write $\phi(y,m)$ for $\phi(y,\omega)$ with similar notation for $\phi^{j,-}$. For $\ell_1 \in A_2$, we rewrite the representation formula as
\be \label{expphi}
             \phi_i(x,m)=\phi_i^{\ell_1,-}(x,m)+E_i^{\ell_1}(x,m)
             \ee
             where
             \bee
             E_i^{\ell_1}(x,m)&=&\int_{\partial D^{\ell_1}(m)} G_{\ell_1}(x,y,m) n^{\ell_1}_i(y) dS\\
             &&\qquad +\int_{\partial D^{\ell_1}(m)} \partial_n G_{\ell_1}(x,y,m) \phi_i(y,m)dS,
             \eee
             where $\partial_n G_{\ell_1}=n^{\ell_1} \cdot \nabla G_{\ell_1}$. 
             In the same way, we have
       $$
             \phi_i(x,m^{\ell_1})=\phi_i^{\ell_1,-}(x,m^{\ell_1})+E_i^{\ell_1}(x,m^{\ell_1}).
             $$
             We recall that the notation $m^{\ell_1}$ means that the $\ell_1$ component of $m$ is replaced by an independent copy. Since there is no inclusion in the cell $\ell_1$, we have $\phi_i^{\ell_1,-}(x,m^{\ell_1})=\phi_i^{\ell_1,-}(x,m)$ and
             \be \label{Dphi}
              \phi_i(x,m)-\phi_i(x,m^{\ell_1})=\Delta E_i^{\ell_1}(x,m)=E_i^{\ell_1}(x,m)-E_i^{\ell_1}(x,m^{\ell_1}).
             \ee
             A similar relation holds for $\phi_i^-:$
             \be \label{Dphi2}
              \phi^-_i(x,m)-\phi^-_i(x,m^{\ell_1})=\Delta E_i^{\ell_1,-}(x,m)=E_i^{\ell_1,-}(x,m)-E_i^{\ell_1,-}(x,m^{\ell_1}),
             \ee
              where
              \bee
              E_i^{\ell_1,-}(x,m)&=&\int_{\partial D^{\ell_1}(m)} G_{\{0,\ell_1\}}(x,y,m) n^{\ell_1}_i(y) dS\\
              &&\qquad + \int_{\partial D^{\ell_1}(m)} \partial_n G_{\{0,\ell_1\}}(x,y,m) \phi^-_i(y,m)dS.
              \eee
              Above, $G_{\{0,\ell_1\}}$ is defined in the same manner as $G_{\ell_1}$ with now inclusions removed in both cells $\ell_1$ and $0$.         Denoting by $\E_{\ell_1}$ integration with respect to $(\theta_{\ell_1},\rho_{\ell_1})$  and $(\theta_{\ell_1}',\rho_{\ell_1}')$ when $\alpha=\theta$ and with respect to $a_{\ell_1}$  and $a_{\ell_1}'$ when $\alpha=a$, we find 
        \be
        \E_{\ell_1} \{\Delta E_i^{\ell_1}\} =  \E_{\ell_1} \{\Delta E_i^{\ell_1,-}\}=0. \label{AV}
        \ee
This relation will be crucially used further.
        


Besides, estimates \fref{estE}-\fref{estE2} of Appendix \ref{phijm} read in this context
        \bea \nonumber
        |E^{\ell_1}(x,m)| &\leq& C (\textrm{dist}(x, D^{\ell_1}(m)))^{-2} (1+\| \phi_i\|_{L^1(\partial D^{\ell_1}(m))})\\
             |\nabla E^{\ell_1}(x,m)| &\leq& C (\textrm{dist}(x,D^{\ell_1}(m)))^{-3} (1+\| \phi_i\|_{L^1(\partial D^{\ell_1}(m))}).  \label{DE}
             \eea
             When $x$ is in the cell around the origin, and since $\ell_1 \in A_2 $ is at least at a distance $\delta \eta^{-1}$ from the origin, it follows that
             \be E^{\ell_1}=O(\eta^2), \qquad\nabla E^{\ell_1}=O(\eta^3),\label{errorE}\ee
             with same relations for $E_i^{\ell_1,-}$, and this for all $m$.

             \paragraph{Step 3: averages.} We now use the results of the previous paragraph to treat $F(A_1 \cup A_2)-F( A_1 \cup (A_2\backslash \{\ell_1 \}))$. There are several similar terms in $F(A)$,
             and we focus on  $$
        \E\big\{M(m) \otimes M( m ^{A_\theta,*})\big\}
        $$
        since the other contributions follow analogously. 
             
             We plug \fref{Dphi} and \fref{Dphi2} with $m=m^{A_\theta,*}$ into $M$ and obtain
             $$
             M(m^{A_\theta,*})=M(m^{A_\theta \backslash{\{\ell_1}\},*})+\Delta M(m^{A_\theta,*}).
             $$
             The error $\Delta M$ contains a term proportional to $(\nabla \Delta E_i^{\ell_1,-}) \Delta E_i^{\ell_1}$, and two terms linear in $\Delta E^{\ell_1,-}_i$ and $ \Delta E_i^{\ell_1}$.  Based on \fref{errorE}, the quadratic term has a contribution of order $O(\eta^5)$. We have then
             \bee
             \Delta M_{\ell,k}(m^{A_\theta,*}) &=&\int_{\partial B(\theta_0(m^{A_\theta,*}),\rho_0(m^{A_\theta,*})} n \cdot \nabla \Delta E^{\ell_1,-}_\ell(y,m^{A_\theta,*})\\
             &&\qquad \times  \left( y_k+\phi_k (y,m^{A_\theta \backslash{\ell_1},*}) \right) dS\\
             &&+\int_{\partial B(\theta_0(m^{A_\theta,*}),\rho_0(m^{A_\theta,*}))} n \cdot (\nabla y_\ell + \nabla \phi^-_\ell (y,m^{A_\theta \backslash{\ell_1},*}))\\
             &&\qquad \times  \Delta E_k^{\ell_1}(y,m^{A_\theta,*}) dS+O(\eta^5).
             \eee

The first term is of order $O(\eta^3)$, and the second one of order $O(\eta^2)$ according to \fref{errorE}. We replace above $\phi_k$ by $\phi_k^{\ell_1,-}$ in the first term, creating an error of order $O(\eta^2)$ following \fref{expphi} and \fref{errorE}, which,  since $\nabla \Delta E^{\ell_1,-}_\ell$ is of order $O(\eta^3)$, produces an overall error of order $O(\eta^5)$. We replace in the same way $\nabla \phi^-_k$ by $\nabla \phi^{\{0,\ell_1\},-}_k$ (i.e. inclusions in cells $0$ and $j$ are removed) in the second term, creating an overall error of order $O(\eta^5)$. We then have
 \bee
             \Delta M_{\ell,k}(m^{A_\theta,*}) &=&\int_{\partial B(\theta_0(m^{A_\theta,*}),\rho_0(m^{A_\theta,*}))} n \cdot \nabla \Delta E_i^{\ell_1,-}(y,m^{A_\theta,*})\\
             &&\qquad \times  \left( y_k+\phi_k^{\ell_1,-} (y,m^{A_\theta \backslash{\ell_1},*}) \right) dS\\
             &&+\int_{\partial B(\theta_0(m^{A_\theta,*}),\rho_0(m^{A_\theta,*}))} n \cdot (\nabla y_\ell + \nabla \phi^{\{0,\ell_1\},-}_\ell (y,m^{A_\theta \backslash{\ell_1},*}))\\
             &&\qquad \times  \Delta E_i^{\ell_1}(y,m^{A_\theta,*}) dS+O(\eta^5)\\
             &:=&\Delta M_{\ell,k}^0(m^{A_\theta,*})+O(\eta^5).
             \eee
   
             Hence,
             $$
             M(m^{A_\theta,*})=M(m^{A_\theta \backslash{\{\ell_1}\},*})+\Delta M^0(m^{A_\theta,*})+O(\eta^5).
$$
            
            The term $\Delta M^0(m^{A_\theta,*})$, of order $O(\eta^2)$, is singled out since, according to \fref{AV}, 
             \be\label{AV2}
             \E_{\ell_1} \{ \Delta M^0(m^{A_\theta,*})\}=0,
             \ee
which will be important further.

We consider now the term $M(\omega)$ and replace in its definition $\phi_k$ by $\phi_k^{\ell_1,-}$, creating an error of order $O(\eta^2)$, and  $\nabla \phi^-_k$ by $\nabla \phi^{\{0,\ell_1\},-}_k$ creating an error of order $O(\eta^3)$. It follows that
             \bee
             M_{\ell,k}(m) &=&\int_{\partial B(\theta_0(m),\rho_0(m))} n \cdot \left(\nabla y_\ell +\nabla \phi_\ell^{\{0,\ell_1\},-} (y,m) \right)\\
             &&\qquad \times  \left( y_k+\phi_k^{\ell_1,-} (y,m) \right) dS+O(\eta^2)\\
             &:=&M^0_{\ell,k} (m)+O(\eta^2).
             \eee
             By construction, $M^0_{\ell,k}$ is independent of $(\theta_{\ell_1},\rho_{\ell_1},a_{\ell_1 })$ and $(\theta'_{\ell_1},\rho'_{\ell_1},a'_{\ell_1 })$, so that
             $$
             \E\{M^0(m) \otimes \Delta M^0(m^{A_\theta,*})\}=0.
             $$
             Hence,
             \bee
             \E\{M(m) \otimes (M(m^{A_\theta,*}-M(m^{A_\theta \backslash \{\ell_1\},*})\}&=&\E\{M^0(m) \otimes \Delta M^0(m^{A_\theta,*})\}+O(\eta^4)\\
             &=&O(\eta^4).
             \eee
This shows that removing the index $\ell_1$ in $A_2$ introduces an error of order $O(\eta^4)$ in $\E\{M(m) \otimes M(m^{A_\theta,*})\}$. The other terms in $F$ are treated in the same manner. The order $O(\eta^4)$ is crucial here for the conclusion, and without exploiting \fref{AV2}, the error would be of order $O(\eta^2)$ which would not be sufficient. 
             
             \paragraph{Step 4: conclusion.} There are $p$ indices to remove from $A_2$ to attain the empty set. We have just shown in the previous section that removing one index yields an error of order $O(\eta^4)$. Since $p \leq n-1-N_\eta=O(\eta^{-3})$, it follows
             This shows that, since $p \leq n=O(\eta^{-3})$
             $$
             F(A_1 \cup A_2)-F( A_1)=O(p \eta^4)=O(\eta).
             $$
             Hence, going back to $\calC^{j,\eta}$, 
             \bee
\frac{2}{a_b^2}\calC^{j,\eta}&=&\frac{1}{n}\sum_{m=0}^{n-1}\sum_{
\tiny{\begin{array}{l}
  A=A_1 \cup A_2, \; |A|=m\\
  A_1 \subset I_\eta,\; A_2 \subset I_\eta^c
\end{array}}
} \frac{1}{C_{m}^{n-1}} F(A_1 \cup A_2)\\
&=&\frac{1}{n}\sum_{m=0}^{n-1}\sum_{
\tiny{\begin{array}{l}
  A=A_1 \cup A_2, \; |A|=m\\
  A_1 \subset I_\eta,\; A_2 \subset I_\eta^c
\end{array}}
} \frac{1}{C_{m}^{n-1}} F(A_1)\\
&&+\frac{1}{n}\sum_{m=0}^{n-1}\sum_{
\tiny{\begin{array}{l}
  A=A_1 \cup A_2, \; |A|=m\\
  A_1 \subset I_\eta,\; A_2 \subset I_\eta^c
\end{array}}
} \frac{1}{C_{m}^{n-1}} (F(A_1 \cup A_2)-F(A_1)).
\eee
Since there are $C_m^{n-1}$ terms in the sum over $A$ above, it follows that this last term is of  order $O(\eta)$. Regarding the first one, we rearrange the sum so that
\begin{align*}
&\frac{1}{n}\sum_{m=0}^{n-1}\sum_{
\tiny{\begin{array}{l}
  A=A_1 \cup A_2, \; |A|=m\\
  A_1 \subset I_\eta,\; A_2 \subset I_\eta^c
\end{array}}
} \frac{1}{C_{m}^{n-1}} F(A_1 )\\
&=
\frac{1}{n}\sum_{q=0}^{N_\eta}\sum_{
\tiny{\begin{array}{l}
  A_1 \subset I_\eta,\; |A_1|=q
 \end{array}}}
                                                              \sum_{r=0}^{n-1-N_\eta}\frac{C^{n-1-N_\eta}_r}{C_{q+r}^{n-1}} F(A_1)\\
  &=
\frac{1}{N_\eta}\sum_{q=0}^{N_\eta}\sum_{
\tiny{\begin{array}{l}
  A_1 \subset I_\eta,\; |A_1|=q
 \end{array}}}
    \frac{1}{C_{q}^{N_\eta-1}} F(A_1 ).
\end{align*}
Above, we used that there are $C^{n-1-N_\eta}_r$ terms in the sum over $A_2$ when $|A_2|=r$, and combinatorial relations to go from the second to the third equality. This ends the proof by rearranging the sum. 
\end{proof}

\medskip 
We investigate now the limit of $\calC^{j,\eta}$ as $\eta \to 0$. For this, let
$$
T_{q}=\sum_{
\tiny{\begin{array}{l}
  A_1 \subset I_\eta,\; |A_1|=q
 \end{array}}}
\frac{1}{C_{q}^{N_\eta-1}} F(A_1)
$$
so that
$$
\calC^{j,\eta}=\frac{a_b^2}{2 N_\eta}\sum_{q=0}^{N_\eta} T_{q}+O(\eta).
$$
We will exploit the fact that
$$
\lim_{N \to \infty}\frac{1}{N} \sum_{q=0}^N c_q = \lim_{q \to \infty} c_q.
$$
As $\eta \to 0$, $N_\eta \to \infty$ and the set $I_\eta$  includes all indices in $\Zm^3 \backslash \{0\}$ when $\eta \to 0$. Hence, $F(A_1)$ becomes $F(\Zm^3 \backslash \{0\})$ in the limit. Since there are $C_{q}^{N_\eta-1}$ terms in the sum in $T_q$, it follows that $\lim_{q \to \infty} T_{q} =F(\Zm^3 \backslash\{0\})$ and as a consequence
\bee 
 \lim_{\eta \to 0}\calC^{j,\eta}&=&\frac{a_b^2}{2} F(\Zm^3 \backslash \{0\})\\
 &=&\frac{a_b^2}{2}\E\Big\{\left(M(m)-M(m^{0_\theta})\right) \otimes \left(M(m^{\Zm^3 \backslash \{0\}_\theta})-M(m^{\Zm^3_\theta})\right)\Big\},
 \eee
 where the notation $m^{\Zm^3 \backslash \{0\}_\theta}$  means that the components $(\theta_j)_{j \in \Zm^3_\star}$ are replaced by an independent copy, while $m^{\Zm^3_\theta}$ means that all $(\theta_j)_{j \in \Zm^3}$ are replaced by an independent copy. With
 \bee\langle M \rangle_{(\theta_0,\rho_0)}  &=&\int_{M_0^{\Zm^3_\star}} M(m) \otimes_{j \in \Zm^3_\star} d \nu_0(\theta_j,\rho_j) \\
 \langle M \rangle &=&\int_{M_0^{\Zm^3}} M(m) \otimes_{j \in \Zm^3} d \nu_0(\theta_j,\rho_j),
 \eee
 where $\langle M \rangle_{(\theta_0,\rho_0)}$ depends on $(\theta_0,\rho_0)$ and $(a_j)_{j \in \Zm^3}$, while $\langle M \rangle$ depends on  $(a_j)_{j \in \Zm^3}$, we find after a direct calculation
 \bee 
 \lim_{\eta \to 0}\calC^{j,\eta}
 &=& a_b^2 \, \E\big \{ (\langle M \rangle_{(\theta_0,\rho_0)} -\langle M \rangle)\otimes (\langle M \rangle_{(\theta_0,\rho_0)}-\langle M \rangle)\big\}.
 \eee
 This term gives the correlation $a_b^2 C^W$ defined in Section \ref{mainresult}. There are other terms in \fref{cov3} coming from the expansion of $\Delta_j^\alpha u_\eta(g,\omega)$.  The analysis is similar to the previous one, and we find for the cross term for $\alpha=\theta$,
 \begin{align*}
 &-\frac{1}{2}\E\Big\{\left(M(m)-M(m^{0_\theta})\right) \left(N(m^{\Zm^3 \backslash \{0\}_\theta})-N(m^{\Zm^3_\theta})\right)\Big\}\\
 &= -  \; \E \big \{ (\langle M \rangle_{(\theta_0,\rho_0)} -\langle M \rangle)\langle (N \rangle_{(\theta_0,\rho_0)}-\langle N\rangle) \big\},
\end{align*}
 which gives the $\calC^{W,N}$ of Section \ref{mainresult}. The terms quadratic in $N$ for $\alpha=\theta$ follow in the same manner and are not detailed. These yield  the $\calC^\theta$ and $\calC^{\theta,*}$ in Section \ref{mainresult}. Remains the correlation term when $\alpha=a$, and we find
\begin{align*}
 &\frac{1}{2}\E\Big\{\left(N(m)-N(m^{0_a})\right) \left(N(m^{\Zm^3 \backslash \{0\}_a,*})-N(m^{\Zm^3_a,*})\right)\Big\}\\
 &=  \E \big \{ (\E \{ N | a_0 \}-\E\{N\})(\E \{N | a_0 \}-\E\{N\} \big\},
\end{align*}
where $\E \{M | a_0 \}$ is the conditional average of $M$ on $a_0$. With the complex conjugate version we obtain $\calC^a$ and $\calC^{a,*}$.

To end the proof of Theorem \ref{th}, it suffices to observe that the only terms in \fref{cov3} that depend on $j$  as $\eta \to 0$ involve $G_{\h,g}(j \eta)$, $u_{\h}(j \eta)$ and their gradients. Then, sums of the form
$$
\eta^3 \sum_{j \in J_\delta^\eta} G_{\h,g}(j \eta) u_{\h}(j \eta)
$$
are Riemann sums for the integral
$$
\int_{\calB_\delta}G_{\h,g}(x) u_{\h}(x) dx,
$$
where $\calB_\delta$ denotes the set of points in $\calB$ at a distance at least of $\delta$ to $\partial \calB$. Sending $\delta$ to 0 then concludes the proof.
  
\section{Conclusion} We have investigated in this work the scaling limit of the centered wavefield $u_\eta-\E\{u_\eta\}$. We obtained the expression of the second-order statistics in the limit $\eta \to 0$. These correspond to those of the solution of the homogenized Helmholtz equation with white noise source terms. There are two contributions in the white noise: one arising from fluctuations in the position and radius of the inclusions, and one from fluctuations in the inverse permittivity. The limiting model can be used for instance in the resolution of some inverse problems involving sea ice.

There are several possible extensions of this work. The next natural step is to establish the Gaussian nature of the scaling limit. This requires additional statistical tools and will be addressed separately. Another one is to develop a fully rigorous mathematical theory, which is by no means an easy task. There are many questions to address, all of which we believe are open in the context of high contrast stochastic homogenization: higher order statistical integrability of the corrector and modified corrector; Nash-Aronson type estimates for gradients of Green's functions in random perforated domains; strong convergence of two-scale expansions. 

\begin{appendix}
\section{Two-scale expansions} \label{2scale}
We derive in this section two-scale expansions for $u_\eta$ and $G_\eta^{j,-}$. 
We follow \cite{felbacq2005negative} and start by writing the following expansion in $\calB$. In order to avoid boundary effects, we consider $x \in \calB$ such that $x$ is at least at a distance $\delta>0$ to the boundary of $\calB$. Close to the boundary, the expansion below has to be modified but this won't be needed in our analysis. With $y=x/\eta$ the fast variable, we have
      \begin{align}
      &u_\eta(x,y)=u_0(x,y)+\eta u_1(x,y)+\eta^2u_2^\eta(x,y) \label{u}\\
&a_\eta \nabla u_\eta=J_\eta= J_0(x,y)+\eta J_1(x,y)+\eta^2 J_2^\eta(x,y) \label{J}.
      \end{align}
     We have then from the Helmholtz equation
         \be
     \nabla \cdot J_\eta+k_0^2 u_\eta=0\label{helmJ}.
      \ee
      With $\nabla=\nabla_x+ \eta^{-1} \nabla_y$, plugging \fref{u}-\fref{J} in \fref{helmJ} and equating like powers of $\eta$ gives
       \begin{align}
      &\textrm{order}\; \eta^{-1}: \nabla_y \cdot J_0=0 \label{or0}\\
         &\textrm{order}\; \eta^{0}: \nabla_x \cdot J_0+\nabla_y \cdot J_1+k_0^2 u_0=0. \label{or1}
       \end{align}
From \fref{J}, we find on $D_\eta(\omega)^c \cap \calB$,
      \begin{align} \label{outD}
        &\textrm{order}\; \eta^{-1}:\quad   a_b \nabla_y u_0=0\\ \label{outD2}
          &\textrm{order}\; \eta^{0}:\quad a_b(\nabla_x u_0+\nabla_y u_1)=J_0,
      \end{align}
and on $D_\eta(\omega)$,
      \begin{align}
        &\textrm{order}\; \eta^{0}: J_0=0,  \label{eqJ}\\
        &\textrm{order}\; \eta^{1}:a_{\eta,\textrm{inc}}(x,\omega) \nabla_y u_0=J_1, \qquad \label{eqJ2}
      \end{align}
We recall that $a_{\eta,\textrm{inc}}(x,\omega)=a_j(\omega)$ for $x \in D^j(\omega)$.
    \paragraph{Continuity conditions.} The fluxes across the boundary of $D_\eta(\omega)$ are continuous, and denoting by $\nabla u_\eta^\pm$ the gradients at the outer/inner boundary of $D_\eta(\omega)$, we have
      $$
      a_{\eta,\textrm{inc}}(x,\omega) \eta^2 n\cdot \nabla u_\eta^-=a_b n \cdot \nabla u^+_\eta,
      $$
      where $n$ is the outward normal to the boundary. 
       In terms of the two-scale expansion, this gives:
      \be \label{BC1}
      \textrm{order}\; \eta^{-1}:n \cdot \nabla_y u_0^+=0\qquad \textrm{order}\; \eta^{0}: n \cdot \nabla_x u_0^+=-n \cdot \nabla_y u_1^+.
      \ee
      Moreover, $u_\eta$ is continuous over $\Rm^3$, and therefore $u_\eta^-=u_\eta^+$. This gives at first order: 
      \be \label{contU}
      u_0^-=u_0^+. 
      \ee
      
      \paragraph{First-order corrector on $D_\eta(\omega)^c$.} Let $\calB_\eta=\{y \in \Rm^3: \eta y \in \calB\}$ and $D_\eta^0=\cup_{j \in J_\eta(\omega)} D^j(\omega)$. Taking the $y$ divergence of \fref{outD2} and using \fref{or0}-\fref{outD} yields
\be \label{equ1}
      \Delta_y u_1=0 \qquad \textrm{on} \qquad D^0_\eta(\omega)^c \cap \calB_\eta.
      \ee
      We have in addition the second equation in \fref{BC1} at the exterior boundary of $D^0_\eta(\omega)$. Recall that $x$ is at a distance $\delta$ to $\partial \calB$, so that $y$ is at least  at  a distance $\delta \eta^{-1}$ to the boundary and boundary effects are ignored on $u_1$. 

      According to \fref{outD}, we have $u_0(x,y)=u_{\rm{h}}(x)$ on $D_\eta(\omega)^c \cap \calB$ for some function $u_{\rm{h}}$ to be determined. We then write $u_1(x,y)= \phi(y) \cdot \nabla u_{\rm{h}}(x)$ with $\phi=(\phi_1,\cdots,\phi_3)$ such that
      $$
      \Delta_y \phi_i=0 \qquad \textrm{on} \quad D(\omega)^c
$$
      and
$$
      n \cdot \nabla_y \phi_i=- e_i \cdot n \qquad \textrm{on} \qquad  \partial D(\omega).
$$
      Above, $e_i$ is a vector of the canonical basis of $\Rm^3$. The corrector $\phi_i$ is obtained as explained in Section \ref{setting}, 
      and the constructed $u_1$ satisfies \fref{equ1} and the second boundary condition in \fref{BC1}.

      In terms of $J_0$, setting
      $$
      J_0(x,y)=a_b\left\{
        \begin{array}{l}(I+(\nabla_y \phi_1(y),\cdots,\nabla_y \phi_3(y))) \nabla_x u_\h(x) \qquad \textrm{on}  \quad (\eta D(\omega))^c\\
          0 \qquad \textrm{on} \quad \eta D(\omega)
      \end{array}
      \right.
      $$
      produces a $J_0$ that satisfies \fref{or0}-\fref{outD2}-\fref{eqJ}. Above, $I$ is the $3\times 3$ identity matrix and $u_\h$ was extended to $\calB^c$ to some function to be determined. 

      We now identify $u_\h$ and need for this to characterize $u_0$ on $D_\eta(\omega)$.
      \paragraph{Homogenized solution.} Since $J_0=0$ on $\eta D(\omega)$ according to definition above,  \fref{or1} becomes $\nabla_y \cdot J_1+k_0^2 u_0=0$ on $D_\eta(\omega)$. Using the latter, and taking the $y$ divergence of \fref{eqJ2} yields
      $$
      a_{\eta,\textrm{inc}}(x,\omega) \Delta_y u_0 + k_0^2 u_0=0, \qquad \textrm{on} \quad D_\eta(\omega)
      $$
      where $a_{\eta,\textrm{inc}}(x,\omega)$ is equal to $a_j(\omega)$ on $D^j_\eta(\omega)$. Accounting for the continuity condition \fref{contU}, we then write $u_0(x,y)=u_{\rm{h}}(x) \Lambda(y,\omega)$ with $\Lambda(y,\omega)$ the solution to, for all $j \in \Zm^3$,
      \be \label{Lamb2}
      a_j(\omega) \Delta_y \Lambda + k_0^2 \Lambda=0, \qquad \textrm{in} \quad D^j(\omega),
      \ee
      with $\Lambda=1$ on the interior boundary of $D_j(\omega)$. The function $\Lambda$ is then extended to 1 on $D(\omega)^c$.

      The homogenized solution $u_{\rm{h}}$ is found by assuming that $u_\h$ is deterministic, by taking the expectation of \fref{or1} and by imposing the closure condition $ \lim_{\eta \to 0 }\E_\calP\{ \nabla_y \cdot J_1(x,x/\eta)\}=0$. We then find,
      $$
   \lim_{\eta \to 0}   \nabla_x \cdot \E_\calP\{ J_0\}+k_0^2 u_{\rm{h}} \E_\calP\{\Lambda\}=0, \qquad \textrm{in}\quad \calB.
      $$
    %
      The flux $J_0$ can be recast as
      $$
      J_0(x,y)=a_b \un_{D(\omega)^c }(y) \left( \nabla_x  u_{\rm{h}}(x)+ \nabla_y u_1(x,y)\right).
      $$
      with
      $$
      \E_\calP\{ J_0(x,y)\}=A_{\rm{eff}} \nabla_x u_{\rm{h}}(x)
      $$
      where the effective matrix $A_{\rm{eff}}$ is defined by
      $$
      A_{\rm{eff}}=a_b \E_\calP\left\{\un_{D(\omega)^c}(y)\big(I+(\nabla_y \phi_1(y),\cdots,\nabla_y \phi_d(y))\big)\right\}.$$
      According to Section \ref{setting}, this expression is equivalent to
      $$
      A_{\rm{eff}}=a_b \int_{\Sigma^c}(I+(v^{1}(\omega),\cdots,v^{3}(\omega))d\calP(\omega).
      $$
We have seen that equation \fref{Lamb2} can be solved exactly resulting in expression for the effective permeability $\mu_{\rm{eff}}$. 
      We then find, at a distance $\delta$ to the boundary of $\cal B$:
      $$
      \nabla_x \cdot A_{\rm{eff}} \nabla_x  u_{\rm{h}}+k_0^2 \mu_{\rm{eff}} u_{\rm{h}}=0.
      $$
      Since $\delta$ is arbitrary, the relation above is satisfied in the entire $\calB$. Outside of $\calB$, we have
      \be \label{out}
      \Delta_x  u_\eta+k_0^2 u_{\eta}=f,
      \ee
      equipped with Sommerfeld radiation condition. Continuity of $u_\eta$ and the flux across $\partial \calB$ yields on $\partial \calB$
      $$
   \lim_{\eta \to 0} u_\eta= u_{\rm{h}}, \qquad   \lim_{\eta \to 0} n \cdot \nabla u_\eta= n \cdot (A_{\rm{eff}} \nabla_x  u_{\rm{h}}).
      $$
      Defining $u^{\textrm{out}}=  \lim_{\eta \to 0} u_\eta$ on $\calB^c$, if follows that  $u^{\textrm{out}}$ satisfies \fref{out} and the continuity conditions above. Extending $u_\h$ by $u^{\textrm{out}}$ to the whole $\Rm^3$ then yields the homogenized system \fref{eqhomo}. 

      \paragraph{Two-scale expansions for $G_\eta^{j,-}$.} Proceeding as the previous section, we find that an expansion of the form
      $$
      G_\eta^{j,-}(x,x',\omega)= \Lambda^{j,-}(x/\eta,\omega) G_{\h}(x,x')+\eta \phi^{j,-}(x/\eta,\omega) \cdot \nabla_x G_{\h}(x,x')+O(\eta^2),
      $$
      cancels the first leading terms in a two-scale expansions in the Helmholtz equation satisfied by $G_\eta^{j,-}$. Above, $\phi^{j,-}$ is the modified corrector introduced in Section \ref{setting} and $\Lambda^{j,-}$ is defined in Section \ref{secSV}. It just remains to identify $G_{\h}(x,x')$. We have for this, at a distance $\delta$ to the boundary for $\cal B$:
\begin{align*}
  &\lim_{\eta \to 0}   \nabla_x \cdot \E_\calP\big\{J_0(x,x/\eta,x')\big\}+k_0^2 G_{\h}(x,x') \E_\calP\{\Lambda^{j,-}(x/\eta)\}=0,
\end{align*}
where the flux $J_0$ is now
$$
 J_0(x,y,x')=a_b \un_{\calD_j(\omega)^c }(y) \left(I+ (\nabla_y \phi_1^{j,-},\cdots, \nabla_y \phi_3^{j,-})(y)\right) \nabla_x  G_{\h}(x,x').
$$
The domain $\calD_j(\omega)$ is defined in Section \ref{setting}. The main difference with the previous section is that $J_0$ is not stationary as a function of $y$, and therefore the average $\E_\calP\big\{J_0\big\}$ depends on $x$. Yet, since the lack of stationarity is due to the absence  of an inclusion in the cell $j$  of volume of order $\eta^3$, the term $\E_\calP\big\{J_0\big\}$ becomes independent of position as $\eta \to 0$. The same holds for $\E_\calP\{\Lambda^{j,-}(y)\}$. To see this, we write
\bee
J_0(x,y,x')&=& a_b \un_{D(\omega)^c }(y) \left(I+ (\nabla_y \phi_1^{j,-},\cdots, \nabla_y \phi_3^{j,-})(y)\right) \nabla_x  G_{\h}(x,x')\\
&&+a_b \un_{D^j(\omega) }(y) \left(I+ (\nabla_y \phi_1^{j,-},\cdots, \nabla_y \phi_3^{j,-})(y)\right) \nabla_x  G_{\h}(x,x')\\
&:=&\widetilde{J}_0(x,y)+\delta J(x,y).
\eee
After integration w.r.t. $x$, the flux $\delta J(x,x/\eta)$ produces a term of order $\eta^3$ since the volume of $\eta D^j(\omega)$ is of order $\eta^3$. Hence, $\delta J(x,x/\eta)$ is negligible in the limit $\eta \to 0$.

     Regarding $\widetilde{J}_0$, we want to replace $\phi^{j,-}$ by $\phi$, and recall for this the estimate \fref{estE2} obtained in Appendix \ref{phijm}:
      \bee
      |(\nabla \phi_i^{j,-}-\nabla \phi_i)(x/\eta)| &\leq& C (\textrm{dist}(x/\eta,\eta D^j(\omega))^{-3} \int_{\eta \partial D^j(\omega)} (1+|\phi_i(y/\eta)|) dS^j(y)\\
      &\leq& C (\textrm{dist}(x/\eta,\eta D^j(\omega))^{-3} \eta^{2}.
      \eee
      Hence, setting $x$ at a distance at least $2 \eta$ from $\eta D^j(\omega)$ produces an error of order $\eta^2$ on $|(\nabla \phi_i^{j,-}-\nabla \phi_i)(x/\eta)|$. The volume of the domain where $x$ is at a distance less than $2 \eta$ to $\eta D^j(\omega)$ is of order $\eta^3$, and the corresponding term becomes negligible after integration w.r.t. $x$. Hence, replacing $\phi^{j,-}$ by $\phi$ in $\widetilde{J}_0$ produces a negligible error in the limit. Collecting all these results shows that
      \begin{align*}
        & \lim_{\eta \to 0}   \E_\calP\big\{J_0(x,x/\eta,x')\big\}\\
        &\qquad = \lim_{\eta \to 0} \E_\calP\big\{a_b \un_{\calD(\omega)^c }(x/\eta) \left(I+ (\nabla_y \phi_1,\cdots, \nabla_y \phi_3)(x/\eta)\right) \big\} \nabla_x  G_{\h}(x,x')\\
      &\qquad = A_\textrm{eff} \nabla_x  G_{\h}(x,x').
      \end{align*}
      In the same way, we find
      $$
       \lim_{\eta \to 0}  \E_\calP\{\Lambda^{j,-}(x/\eta)\}=\mu_{\textrm{eff}},
       $$
       so that, as expected,
       $$
      \nabla_x \cdot A_{\rm{eff}} \nabla_x G_{\h}(x,x') +k_0^2 \mu_{\rm{eff}} G_{\h}(x,x')=\delta(x-x') \qquad \textrm{on} \quad \calB.
      $$
The function $G_{\h}(x,x')$ is then extended to $\Rm^3$ accordingly to obtain \fref{homoGF}.

  \section{Results on Green's functions} \label{secG} Consider $\calD=\cup_{n=1}^\infty B_n \subset \Rm^3$, where the $(B_n)$ form a collection of non-intersecting open balls with radii in $[R_1, R_2]$, and with boundaries separated at least by distance $d>0$. 
  Let $G$ be the Green's function verifying
               $$
               -\Delta_x G(x,y)=\delta(x-y),\qquad x,y \in \calD^c, \qquad \partial_n G(x,y)=0 \qquad \textrm{on} \quad \partial \calD,
               $$
               with $\lim_{|x| \to \infty} G=0$. Above, $\partial_n G$ is the normal derivative on $\partial \calD$. Nash-Aronson estimates for the heat equation, see \cite{jikov2012homogenization}, Chapter 8, show, after integration over the time variable, that there is a constant
               $C>0$ depending only on $R_1,R_2$, such that
               \be \label{NA}
               0 \leq G(x,y) \leq C|x-y|^{-1}, \qquad \textrm{for all}\quad x,y \in \calD^c.
               \ee
               The above estimate applies to the case considered in this work  where $\calD=D(\omega)$, and we still denote by $G$ the corresponding Green's function. Gradient estimates are obtained using large-scale regularity results, and a rigorous theory is available in \cite{armstrong2019quantitative} in the case of stochastic homogenization of classical elliptic equations without the large contrast property. To the best of our knowledge, there is no such theory in the large contrast case, so we provide informal arguments leading to the expected estimates.

               We denote by $\mathfrak{G}_\h$ the homogenized Green's function satisfying
                $$
               -\nabla_x A_{\textrm{eff}} \nabla_x \frakG_\h(x,y)=\delta(x-y),\qquad x,y \in \Rm^3,
               $$
               with  $\lim_{|x| \to \infty} \frakG_\h=0$. When $x \neq y$, $\frakG_\h$ is a so-called a-harmonic function, see \cite{armstrong2019quantitative}, and behaves, along with its derivatives, in the same way as $|x-y|^{-1}$, namely
               \be \label{aharmo}
               |\frakG_\h(x,y)|+ |x-y| |\nabla_x \frakG_\h(x,y)|+|x-y|^2 |\nabla_y \nabla_x \frakG_\h(x,y)| \leq C |x-y|^{-1}.
               \ee

               Let $x_0,y \in \calD^c$ such that $4R=|x_0-y| \gg 1$, and consider a smooth function $\chi$ such that $\chi=1$ on $B_R(x_0)$ and $\chi=0$  on $B_{2R}(x_0)$. We have $|\partial_{x_i} \chi|+ R |\partial^2_{x_i x_j} \chi| \leq C R^{-1}$.  We have then the expression
               $$
               \chi(x)G(x,y)=\int_{\calD^c} \left( 2 \nabla \chi(x') \cdot \nabla_x G(x',y)+\Delta \chi(x') G(x',y)\right) G(x',x) dx',
               $$
               so that, when $x \in B_{R/2}(x_0)$,
               $$
               \partial_{x_i} G(x,y)=\int_{\calD^c} \left( 2 \nabla \chi(x') \cdot \nabla_x G(x',y)+\Delta \chi(x') G(x',y)\right) \partial_{x_i} G(x',x) dx'.
               $$
               For $\eps=R^{-1}$, consider the shifted rescaled Green's function $G^\eps(x,y)=G(R+R x,y)$, where $x$ is such that $x+R$ (with an abuse of notation $x+R$ means that all components of $x$ are shifted by $R$) belongs to the intersection the annulus of inner and outer radii $0$ and $1$ and $R^{-1} \calD^c$. We denote that set by $A_R$. We have then, for $x \in B_{R/2}(x_0)$.
               $$
               |\partial_{x_i} G(x,y)| \leq C R^2 \int_{A_R} |(\nabla_x G^\eps)(x',y)| |(\partial_{x_i} G^\eps)(x',x)| dx'+C \int_{A_R} |(\partial_{x_i} G^\eps)(x',x)| dx'.
                 $$
                 Above, we used \fref{NA} to control $G(x',y)$, the estimates on $\chi$ stated above, and the fact that $|x'-y|\geq C R$ when $x' \in B_{2R}(x_0)$. After averaging over $A_R$ (which is of size $O(1)$), we can use the two-scale expansion of $\nabla _x G^\eps$, which reads, for $x' \in A_R$,
                 $$
                 \nabla_x G^\eps(x',y) \simeq Q(R+Rx',\omega) \nabla_x \frakG(R+x',y).
                 $$
                 The term $Q$ is defined in \fref{defAB}, and the two-scale expansion of $G$ is obtained in a similar way as that of $u_\eta$. Using \fref{aharmo}, we have, since $|R+x'-y| \geq C R$ when $x' \in A_R$ and $|R+x'-x| \geq C R$ when $x' \in A_R$ and $x \in B_{R/2}(x)$,  $|\nabla_x G^\eps(x',y)|+ |\nabla_x G^\eps(x',x)|=O(R^{-2})$, and as a consequence
                  $$|\partial_{x_i} G(x,y)| =O(R^{-2})=O(|x_0-y|^{-2})=O(|x-y|^{-2}), \qquad x \in B_{R/2}(x_0).$$
                  Estimates on $\nabla_y \nabla_x G$ follow in the same manner, and we find $|\nabla_y \nabla_x G(x,y)|=O(|x-y|^{-3})$.

                  When removing inclusions in $D(\omega)$, say in the cell $j$, we use the two-scale expansion
                  $$
                 \nabla_x G^\eps(x',y) \simeq Q^{j,-}(R+Rx',\omega) \nabla_x \frakG(R+x',y),
                 $$
                 where $Q^{j,-}$ is defined in \fref{2sGum}, and the estimates follow in the same way of above. The two-scale expansion follows the derivation of that of $G^{j,-}_\eta$, see Appendix \ref{2scale}.

  \section{The modified corrector $\phi^{j,-}$} \label{phijm}
  We show in this section how to construct a solution to \fref{corr1m}-\fref{corr2m}-\fref{corr3m} and show it is unique up to an additive constant. We also derive a helpful representation formula for the modified corrector.

  \paragraph{Existence.} Let $\chi$ be a smooth function equal to one outside of the cell $j$ and equal to zero on the closure of $D^j(\omega)$. For $i=1,2,3$, we look for a solution under the form $\phi_i^{j,-}(y,\omega)=\chi \phi_i(y,\omega)+u$. Based on \fref{corr1m}-\fref{corr2m}, the function $u$ verifies
 $$
      \Delta u=-2 \nabla \chi \cdot \nabla \phi_i-\phi_i \Delta \chi \qquad \textrm{on} \quad \calD_j(\omega)^c
      $$
      and, for $\ell \neq j$,
      $$
      n^\ell \cdot \nabla u=0 \qquad \textrm{on} \qquad  \partial D^\ell(\omega).
      $$
      According to \fref{corr3m}, $u$ should satisfy the following condition at infinity:
       \be \label{infu}
      \lim_{|y| \to \infty} \un_{\calD_j(\omega)^c} \nabla u=0.
      \ee

Such a  function $u$ can be obtained by using the perforated Green's function studied in the previous section. Let then $G_j$ be defined by, for $y \in \calD_j(\omega)^c$,
      $$
      \Delta_x G_j(x,y) = - \delta(x-y) \qquad \textrm{on} \quad \calD_j(\omega)^c
      $$
      and, for $\ell \neq j$,
      $$
      n^\ell \cdot \nabla_xG_j(x,y)=0 \qquad \textrm{on} \qquad  \partial D^\ell(\omega),
      $$
      with $\lim_{|x| \to \infty} \un_{\calD_j(\omega)^c} G_j=0$. Then $u$ reads, for $x \in \calD_j(\omega)^c$, 
      $$
      u(x)=\int_{ \calD_j(\omega)^c} (2 \nabla \chi \cdot \nabla \phi_i+\phi_i \Delta \chi)(y) G_j(x,y) dy. 
      $$
      Since $\nabla \chi$ and $\Delta \chi$ are supported in the cell $j$ and $\nabla G_j(x,y) \to 0$ as $|x| \to \infty$ according to Appendix \ref{secG}, it follows that $u$ verifies \fref{infu}. The function $\phi_i+\chi u$ then gives a solution to \fref{corr1m}-\fref{corr2m}-\fref{corr3m}.

     \paragraph{Representation formula.}     We will use a different representation formula for $\phi^{j,-}$. We find first, after an integration by parts:
  \bee
       u(x)&=&\int_{ \calD_j(\omega)^c} \nabla \chi \cdot \nabla \phi_i  G_j(x,y) dy-
       \int_{ \calD_j(\omega)^c} \phi_i \nabla \chi \cdot \nabla_y G_j(x,y) dy. 
       \eee
       Setting for $\chi$ a sequence of smooth functions converging to $\un_{D( \omega)^c}$, we find in the limit
 \bee
       u(x)&=&\int_{ \partial D^j( \omega)} n^j \cdot \nabla \phi_i(y, \omega) \, G_j(x,y) dS-
       \int_{ \partial D^j( \omega)} \phi_i(y, \omega) \, n^j \cdot \nabla_y G_j(x,y) dS. 
       \eee
In terms of $\phi_i^{j,-}$, this finally gives, on $D(\omega)^c$:
       \begin{align} \nonumber
         &\phi_i^{j,-}(x,\omega)-\phi_i(x, \omega)\\\nonumber
         & =\int_{ \partial D^j( \omega)} n^j \cdot \nabla \phi_i(y, \omega) \, G_j(x,y) dS-
           \int_{ \partial D^j( \omega)} \phi_i(y, \omega) \, n^j \cdot \nabla_y G_j(x,y) dS\\\label{repr}
         &=-\int_{ \partial D^j( \omega)} n^j_i \, G_j(x,y) dS-
           \int_{ \partial D^j( \omega)} \phi_i(y, \omega) \, n^j \cdot \nabla_y G_j(x,y) dS.
       \end{align}
       Above, we used that $n^j \cdot \nabla \phi_i(y, \omega)= -n_i^j$ on $D^j( \omega)$. Denote by $-E(x)$ the term in \fref{repr}. Since $\int_{ \partial D^j( \omega)} n^j_i (y)dS(y)=0$, the first term in $E$ can be recast as
       $$
       \int_{ \partial D^j( \omega)} n^j_i \, (G_j(x,y)-G_j(x,\theta_j(\omega)+j)) dS.
       $$
       We recall that $D^j(\omega)$ is centered at $\theta_j(\omega)+j$. Let $x \in D(\omega)^c$, $y \in \partial D^j( \omega)$ and let $e=(y-\theta_j(\omega)-j)/|y-\theta_j(\omega)-j|$. Denote by $\partial_e G_j=e \cdot \nabla_y G_j(x,y)$ the directional derivative along $e$. There exists a $y^*$ in the segment joining $y$ and $\theta_j(\omega)+j$ such that
       $$
       G_j(x,y)-G_j(x,\theta_j(\omega)+j)= |y-\theta_j(\omega)-j| \partial_e G(x,y^*).
       $$
       According to the gradient estimate on $G_j$ given in Appendix \ref{secG}, we find
       \bee
       |G_j(x,y)-G_j(x,\theta_j(\omega)+j)| 
       &\leq& C \rho_j(\omega) |x-y^*|^{-2} \\
       &\leq& C \textrm{dist}(x,D^j(\omega))^{-2}.
       \eee
       This, and using again the gradient estimate on $G_j$ finally yields, for $x \in D(\omega)^c$,
       \be \label{estE}
       |E(x)| \leq C \textrm{dist}(x,D^j(\omega))^{-2} (1+ \|\phi_i\|_{L^1(\partial D^j( \omega))}).
       \ee
       A similar calculation shows that
        \be \label{estE2}
       |\nabla E(x)| \leq C \textrm{dist}(x,D^j(\omega))^{-3} (1+ \|\phi_i\|_{L^1(\partial D^j( \omega))}).
       \ee

       \paragraph{Uniqueness.} We now turn to the proof of uniqueness up to an additive constant. Let $\varphi$ be the difference between two possible solutions, which verifies $\Delta \varphi=0$ on $ \calD_j(\omega)^c$, $n^\ell \cdot \nabla \varphi=0$ on  $\partial D^\ell(\omega)$ for $\ell \neq j$, and $\lim_{|x| \to \infty} \un_{\calD^j(\omega)^c} \nabla \varphi=0$. Following Rellich estimates, see \cite{ammari2007polarization}, Corollary 2.20, it follows that the whole gradient $\nabla \varphi$ vanishes on $\partial D^\ell(\omega)$ in addition to the normal derivative. We apply now the maximum principle to $\partial_{x_i} \varphi$, $i=1,2,3$. For $R>0$, let $C_R$ be a smooth contour in $\calD_j(\omega)^c$ enclosing the ball of radius $R$ centered at the origin. The maximum principle implies that a non-vanishing maximum of $\partial_{x_i} \phi$ can only be achieved on $C_R$, unless $\partial_{x_j}\phi$ is constant. Since $\lim_{|x| \to \infty} \un_{\calD^j(\omega)^c} \partial_{x_i} \varphi=0$, this shows that $\partial_{x_i }\varphi$ is constant for $i=1,2,3$, and that constant has to be zero since $\partial_{x_i} \varphi =0$ on $\partial D^\ell(\omega)$. Hence $\varphi$ is constant and there is just one solution up to an additive constant. 
    \end{appendix}
    \bibliographystyle{siam}
  \bibliography{../bibliography} 
\end{document}